\documentclass{amsart}

 \usepackage{amsmath}
\usepackage{amsfonts}
\usepackage{bbm}
\usepackage[dvips]{graphicx}
\usepackage{color}
\usepackage{float} 
\usepackage{epsfig} 
\usepackage{fullpage}
\usepackage{url}

\newcommand{\B}{{\mathcal B}}

\newcommand{\transp}{{\scriptscriptstyle \top}}
\newtheorem{thm}{\bf{Theorem}}[section]
\newtheorem{lemma}[thm]{\bf{Lemma}}

\newtheorem{prop}[thm]{\bf{Proposition}}

\newtheorem{rem}[thm]{\bf{Remark}}

\begin{document}

\title[Convergence of SBDM for risk-averse stochastic convex programs]{Convergence analysis of sampling-based decomposition methods for risk-averse multistage stochastic convex programs}

\author{Vincent Guigues}

\address{Vincent Guigues:
            FGV/EMAp, Praia de Botafogo, Rio de Janeiro, Brazil, {\tt vguigues@fgv.br}}

\begin{abstract} 
We consider a class of sampling-based decomposition
methods to solve risk-averse multistage stochastic convex programs.
We prove a formula for the computation of the cuts
necessary to build the outer linearizations of the recourse functions.
This formula can be used to obtain an efficient implementation of Stochastic Dual Dynamic Programming applied to convex nonlinear problems.
We prove the almost sure convergence of these decomposition methods when the relatively complete
recourse assumption holds. We also prove the almost sure convergence of these algorithms when applied
to risk-averse multistage stochastic linear programs that do not satisfy the relatively complete
recourse  assumption. The analysis is first done assuming the underlying stochastic process
is interstage independent and discrete, with a finite set of possible realizations at each stage.
We then indicate two ways of extending the methods and convergence analysis to the case
when the process is interstage dependent. 
\end{abstract}

\keywords{Stochastic programming \and  Risk-averse optimization \and Decomposition algorithms \and Monte Carlo sampling \and Relatively complete recourse \and SDDP}

\maketitle 

\par AMS subject classifications: 90C15, 90C90.

\section{Introduction}

Multistage stochastic convex optimization models have become a standard tool
to deal with a wide range of engineering problems in which one has to make
a sequence of decisions, subject to random costs and constraints, that arise from observations of a stochastic process.
Decomposition methods are popular solution methods to solve such problems.
These algorithms are based on dynamic programming equations and 
build outer linearizations of the recourse functions, assuming that the
realizations of the stochastic process over the optimization period can be
represented by a finite scenario tree.
Exact decomposition methods such as the Nested Decomposition (ND) algorithm 
\cite{birge-louv-book}, \cite{birgemulti}, compute cuts at each iteration for
the recourse functions at all the nodes of the scenario tree.
However, in some applications, the number of scenarios may become so large that these
exact methods entail prohibitive computational effort. 

Monte Carlo sampling-based algorithms constitute an interesting alternative in such
situations. For multistage stochastic
linear programs (MSLP) whose number of immediate descendant nodes is small but with many
stages, Pereira and Pinto \cite{pereira} propose to sample in the forward pass of
the ND. This sampling-based variant of the ND is the so-called Stochastic Dual Dynamic
Programming (SDDP) algorithm, which has been the object of several
recent improvements and extensions \cite{shapsddp}, \cite{philpmatos}, \cite{guiguesrom10}, \cite{guiguesrom12}, \cite{guiguescoap2013},
\cite{kozmikmorton}.

In this paper, we are interested in the convergence of SDDP and related algorithms for risk-averse
multistage stochastic convex programs (MSCP).
A convergence proof of an enhanced variant of SDDP, the Cutting-Plane and Partial-Sampling (CUPPS) algorithm,
was given in \cite{chenpowell99} for risk-neutral multistage stochastic linear programs with uncertainty in the right-hand side only.
For this type of problems, the proof was later extended to a larger class of algorithms in  \cite{philpot2}, \cite{philpot}.
These proofs are directly applicable to show the convergence of SDDP applied to the risk-averse models introduced in
\cite{guiguesrom10}.
Finally, more recently, Girardeau et al. proved the convergence of a class of sampling-based
decomposition methods to solve some risk-neutral multistage stochastic convex programs \cite{lecphilgirar12}.
We extend this latter analysis in several ways:
\begin{itemize}
\item[(A)] The model is risk-averse, based on dynamic programming equations expressed in terms
of conditional coherent risk functionals.
\item[(B)] Instead of using abstract sets, the dynamic constraints are expressed using
equality and inequality constraints, a formulation needed when the algorithm is implemented
for a real-life application. Regarding the problem formulation, the dynamic
constraints also depend on the full history of decisions instead of just the previous decision.
As a result, the recourse functions also depend on the the full history of decisions and the formulas of the optimality cuts
for these functions and of the feasibility cuts built by the traditional implementation of SDDP (where recourse functions depend on the previous decision only) 
need to be updated, see Algorithms 1, 2, and 3.
\item[(C)] The argument $x_{1:t-1}=(x_1,\ldots,x_{t-1})$ of the recourse function $\mathcal{Q}_t$ for stage $t$
takes values in $\mathbb{R}^{n(t-1)}$ (see Section \ref{framework} for details). 
To derive cuts for this function, we need the description of the subdifferential of
a lower bounding convex function which is the value function of a convex problem.
For that, proceeding as in \cite{lecphilgirar12}, $n(t-1)$ additional variables $z^t \in \mathbb{R}^{n(t-1)}$
and the $n(t-1)$ constraints $z^t = (x_1^\transp,\ldots,x_{t-1}^\transp)^\transp$ would be added. With the 
argument $x_{1:t-1}$ of the value function appearing only in the right-hand side of linear constraints,
\cite{lecphilgirar12} then uses the (known) formula of the subdifferential
of a value function whose argument is the right-hand side of linear constraints.
On the contrary, we derive in Lemma \ref{dervaluefunction}
a formula for the subdifferential of the value function of a convex problem
(with the argument of the value function in both the objective and nonlinear constraints)
that does not need the introduction of additional variables and
constraints.
We believe that this lemma is a key tool for the implementation of SDDP applied to convex problems
and is interesting per-se 
since subgradients of value functions of convex problems are
computed at a large number of points at each iteration of the algorithm. 
The use of this formula should speed up each iteration.
We are not aware of another paper proving this formula.
\item[(D)] A separate convergence proof is given for the case of interstage independent processes in which cuts
can be shared between nodes of the same stage, assuming relatively complete recourse. The way to extend the algorithm and convergence proof to 
solve MSLPs that do not satisfy the relatively complete recourse assumption 
and to solve interstage dependent MSCPs
is also discussed.
\item[(E)] It is shown that the optimal value of the approximate first stage problem
converges almost surely to the optimal value of the problem and that almost surely
any accumulation point of the sequence of approximate first stage solutions
is an optimal solution  of the first stage problem.
\end{itemize}
However, we use the traditional sampling process for SDDP (\cite{pereira}), which is less general
than the one from \cite{lecphilgirar12}. 
From the convergence analysis, we see that the main ingredients on which the convergence of SDDP relies
(both in the risk-averse and risk-neutral settings) are the following:
\begin{itemize}
\item[(i)] the decisions belong almost surely to compact sets.
\item[(ii)] The recourse functions and their lower bounding approximations
are convex Lipschitz continuous on some sets. The subdifferentials of these functions
are bounded on these sets.
\item[(iii)] The samples generated
along the iterations
are independent and at
 each stage, conditional to the history of the process, the number of possible
realizations of the process is finite.
\end{itemize}
Since the recourse functions are expressed in terms of value functions of convex optimization
problems, it is useful to study properties of such functions.
This analysis is done in Section \ref{propvaluefunction} where we provide a formula for the
subdifferential of the value function of a convex optimization problem 
as well as conditions ensuring the continuity of this function and the boundedness
of its subdifferential.
Section \ref{framework} introduces the class of problems and decomposition algorithms we consider
and prepares the ground showing (ii) above. Section \ref{sec:relativelcr} shows the convergence
of these decomposition algorithms
for interstage independent processes when relatively complete recourse holds.
In Section \ref{sec:withoutrelativelcr}, we explain how to extend the algorithm and convergence analysis
for the special case of multistage stochastic linear programs that do not satisfy the relatively complete recourse
assumption. Finally, while Sections \ref{framework}-\ref{sec:withoutrelativelcr} deal with interstage independent processes,
Section \ref{idp} establishes the convergence when the process is interstage dependent.

We use the following notation and terminology:
\begin{itemize}
\item The tilde symbol will be used to represent realizations of random variables:
for random variable $\xi$, $\tilde \xi$ is a realization of $\xi$.
\item For vectors $x_1,\ldots,x_m \in \mathbb{R}^{n}$, 
we denote by $[x_1, \ldots,x_m]$ the $n\small{\times}m$ matrix
whose $i$-th column is the vector $x_i$.
\item For matrices $A, B$, we denote the matrix $\left(\begin{array}{c}A\\B\end{array}\right)$ by $[A;B]$
and the matrix $\left(A \;B\right)$ by $[A, B]$
\item For sequences of $n$-vectors $(x_t)_{t \in \mathbb{N}}$ and $t_1 \leq t_2 \in \mathbb{N}$, 
$x_{t_{1}:t_{2}}$ will represent, depending on the context,
\begin{itemize}
\item[(i)] the Cartesian product $(x_{t_1 }, x_{t_1 +1},\ldots,x_{t_2}) \in \underbrace{\mathbb{R}^{n} {\small{\times}} \ldots {\small{\times}} \mathbb{R}^{n}}_{t_2-t_1+1 \mbox{ times}}$ or
\item[(ii)] the vector 
$[x_{t_1 }; x_{t_1 +1};\ldots;x_{t_2}] \in \mathbb{R}^{n(t_2 - t_1 + 1)}$.
\end{itemize}
\item The usual scalar product in $\mathbb{R}^n$ is denoted by $\langle x, y\rangle = x^\transp y$ for $x, y \in \mathbb{R}^n$.
The corresponding norm is $\|x\|=\|x\|_2=\sqrt{\langle x, x \rangle}$.

\item $\mathbb{I}_{A}(\cdot)$ is the indicator function of the set $A$:
$$
\mathbb{I}_{A}(x) :=
\left\{
\begin{array}{ll}
0, & \mbox{if }x \in A,\\
+\infty, & \mbox{if }x \notin A.
\end{array}
\right.
$$
\item $\mbox{Gr}(f)$ is the graph of multifunction $f$. 
\item $A^*=\{x : \langle x, a \rangle \leq 0, \forall a \in A\}$ is the polar cone of $A$.
\item $\mathcal{N}_A(x)$ is the normal cone to $A$ at $x$.
\item $\mathcal{T}_A(x)$ is the tangent cone to $A$ at $x$.
\item $\mbox{ri}(A)$ is the relative interior of set $A$.
\item $\mathbb{B}_n$ is the unit ball $\mathbb{B}_n=\{x \in \mathbb{R}^n : \|x\| \leq 1 \}$ in $\mathbb{R}^n$.
\item dom($f$) is the domain of function $f$.
\end{itemize}

\section{Some properties of the value function of a convex optimization problem} \label{propvaluefunction}

We start providing a representation of the subdifferential of the value function of a convex optimization problem.
This result plays a central role in
the implementation and convergence analysis of SDDP applied to convex problems and will be used in the sequel.

Let $\mathcal{Q}: X\rightarrow {\overline{\mathbb{R}}}$, be the value function given by
\begin{equation} \label{vfunctionq}
\mathcal{Q}(x)=\left\{
\begin{array}{l}
\inf_{y \in \mathbb{R}^{n}} \;f(x,y)\\
y \in S(x):=\{y\in Y \;:\;Ax+By=b,\;g(x,y)\leq 0\}.
\end{array}
\right.
\end{equation}
Here, $A$ and $B$ are matrices of appropriate dimensions, and
$X \subseteq \mathbb{R}^m$ and $Y \subseteq \mathbb{R}^n$ are nonempty, compact, and convex
sets. Denoting by
\begin{equation}\label{epsfatten}
X^\varepsilon := X + \varepsilon  \mathbb{B}_m
\end{equation}
the $\varepsilon$-fattening of the set $X$, we make the following assumption (H):
\begin{itemize}
\item[1)] $f:\mathbb{R}^m \small{\times} \mathbb{R}^n \rightarrow \mathbb{R} \cup \{+\infty\}$ is lower semicontinuous, proper, and convex.
\item[2)] For $i=1,\ldots,p$, the $i$-th component of function
$g(x,y)$ is a convex lower semicontinuous function
$g_i:\mathbb{R}^m \small{\times} \mathbb{R}^n \rightarrow \mathbb{R} \cup \{+\infty\}$.
\item[3)] There exists $\varepsilon>0$ such that $X^{\varepsilon}{\small{\times}}  Y \subset \mbox{dom}(f)$.
\end{itemize}

Consider the Lagrangian dual problem 
\begin{equation}\label{dualpb}
\displaystyle \sup_{(\lambda, \mu) \in \mathbb{R}^q \small{\times} \mathbb{R}_{+}^{p} }\; \theta_{x}(\lambda, \mu)
\end{equation}
for the dual function
$$
\theta_{x}(\lambda, \mu)=\displaystyle \inf_{y \in Y} \;f(x,y) + \lambda^\transp (Ax+By-b) + \mu^\transp g(x,y).
$$
We denote by $\Lambda(x)$ the set of optimal solutions of the  dual problem \eqref{dualpb}
and we use the notation
$$
\mbox{Sol}(x):=\{y \in S(x) : f(x,y)=\mathcal{Q}(x)\}
$$
to indicate the solution set to \eqref{vfunctionq}.

It is well known that under Assumption (H), $\mathcal{Q}$ is convex and
if $f$ is uniformly convex then $\mathcal{Q}$ is uniformly convex too.
The description of the subdifferential of $\mathcal{Q}$ is given in
the following lemma:
\begin{lemma}\label{dervaluefunction} 
Consider the value function $\mathcal{Q}$ given by \eqref{vfunctionq} and take $x_0 \in X$
such that $S(x_0)\neq \emptyset$.
Let 
$$
C_1 = \Big\{(x,y) \in \mathbb{R}^{m}{\small{\times}}\mathbb{R}^{n} \;:\; Ax+By=b\Big\}
 \mbox{ and }C_2 = \Big\{(x,y) \in \mathbb{R}^{m}{\small{\times}}\mathbb{R}^{n} \;:\;g(x,y)\leq 0 \Big\}.
$$
Let Assumption (H) hold and
assume the Slater-type constraint qualification condition:
$$
there \;exists\; (\bar x, \bar y) \in X{\small{\times}}\emph{ri}(Y) \mbox{ such that } (\bar x, \bar y) \in C_1 \mbox{ and } (\bar x, \bar y) \in \emph{ri}(C_2).
$$
Then $s \in \partial \mathcal{Q}(x_0)$ if and only if
\begin{equation}\label{caractsubQ}
\begin{array}{l}
(s, 0) \in  \partial f(x_0, y_0)+\Big\{[A^\transp; B^\transp ] \lambda \;:\;\lambda \in \mathbb{R}^q\Big\}\\
\hspace*{1.2cm}+ \Big\{\displaystyle \sum_{i \in I(x_0, y_0)}\; \mu_i \partial g_i(x_0, y_0)\;:\;\mu_i \geq 0 \Big\}+ \Big\{\{0\}  \small{\times} \mathcal{N}_{Y}(y_0)  \Big\},
\end{array}
\end{equation}
where $y_0$ is any element in the solution set \mbox{Sol}($x_0$), and with
$$I(x_0, y_0)=\Big\{i \in \{1,\ldots,p\} \;:\;g_i(x_0, y_0) =0\Big\}.$$

In particular, if $f$  and $g$ are differentiable, then 
$$
\partial \mathcal{Q}(x_0)=\Big\{  \nabla_x f(x_0, y_0)+ A^\transp \lambda + \sum_{i \in I(x_0, y_0)}\; \mu_i \nabla_x g_i(x_0, y_0)\;:\; (\lambda, \mu) \in \Lambda(x_0) \Big\}.
$$
\end{lemma}
\begin{proof}
Observe that
$$
\mathcal{Q}(x)=
\left\{
\begin{array}{l}
\inf \; f(x,y)+\mathbb{I}_{\mbox{Gr}(S)}(x,y)\\
y \in \mathbb{R}^{n}
\end{array}
\right.
$$
where $\mathbb{I}_{\mbox{Gr}(S)}$ is the indicator function of the set
$$
\begin{array}{l}
\mbox{Gr}(S):=\Big\{(x,y) \in \mathbb{R}^{m}{\small{\times}}\mathbb{R}^{n} \;:\; Ax+By=b,\;g(x,y)\leq 0,\;y \in Y \Big\}=  C_1 \bigcap C_2 \bigcap \mathbb{R}^m \small{\times} Y.
\end{array}
$$
Using Theorem 24(a) in Rockafellar \cite{rock74}, we have
\begin{equation}\label{subdiffvaluef1}
\begin{array}{lll}
s \in \partial \mathcal{Q}(x_0) &\Leftrightarrow &(s,0) \in \partial (f+\mathbb{I}_{\mbox{Gr}(S)})(x_0,y_0)\\
&\Leftrightarrow &(s, 0)   \in \partial f(x_0,y_0) +  \mathcal{N}_{\mbox{Gr}(S)}(x_0, y_0). \;\;(a)
\end{array}
\end{equation}
For equivalence \eqref{subdiffvaluef1}-(a), we have used the fact that $f$ and $\mathbb{I}_{\mbox{Gr}(S)}$ are proper,
finite at $(x_0, y_0)$, and 
\begin{equation}\label{interdomnempty}
\mbox{ri}(\mbox{dom}(f)) \cap \mbox{ri}(\mbox{dom}(\mathbb{I}_{\mbox{Gr}(S)})) \neq \emptyset.
\end{equation}
The set $\mbox{ri}(\mbox{dom}(f)) \cap \mbox{ri}(\mbox{dom}(\mathbb{I}_{\mbox{Gr}(S)}))$ is nonempty because
it contains the point $(\bar x, \bar y)$:
$$
\begin{array}{lll}
(\bar x, \bar y) \in C_1 \cap \mbox{ri}(C_2) \cap \mathbb{R}^m \small{\times}\mbox{ri}(Y) &= &\mbox{ri}(C_1) \cap \mbox{ri}(C_2) \cap \mathbb{R}^m \small{\times}\mbox{ri}(Y)\\
 &= & \mbox{ri}(C_1 \cap C_2 \cap \mathbb{R}^m \small{\times} Y) = \mbox{ri}(\mbox{dom}(\mathbb{I}_{\mbox{Gr}(S)})),\\
(\bar x, \bar y) \in X\small{\times}\mbox{ri}(Y) \subseteq \mbox{ri}(X^\varepsilon) \small{\times} \mbox{ri}(Y)& = &\mbox{ri}(X^\varepsilon \small{\times} Y) \stackrel{(H)}{\subseteq} \mbox{ri}(\mbox{dom}(f)).
\end{array}
$$
Using the fact $C_1$ is an affine space and $C_2$ and $Y$ are closed and convex sets
such that $(\bar x, \bar y) \in \mbox{ri}(C_2) \cap \mbox{ri}(\mathbb{R}^m  \small{\times} Y) \cap C_1 \neq \emptyset$, we have 
$$
\mathcal{N}_{\mbox{Gr}(S)}(x_0, y_0)=\mathcal{N}_{C_1}(x_0,y_0)+\mathcal{N}_{C_2}(x_0,y_0)+\mathcal{N}_{\mathbb{R}^{m}\small{\times}Y}(x_0,y_0).
$$
But $\mathcal{N}_{\mathbb{R}^{m}\small{\times}Y}(x_0,y_0)=\{0\}\small{\times}\mathcal{N}_{Y}(y_0)$ and 
standard calculus on normal and tangent cones shows that 
$$
\begin{array}{lll}
\mathcal{T}_{C_1}(x_0,y_0)&=&\{(x,y) \;:\;Ax+By=0\}=\mbox{Ker}([A,B]),\\
\mathcal{N}_{C_1}(x_0, y_0)&=&\mathcal{T}_{C_1}^{*}(x_0,y_0)=(\mbox{Ker}([A, B])^{\perp}\\
&=&\mbox{Im}[A^\transp; B^\transp]=\Big\{[A^\transp; B^\transp ] \lambda \;:\;\lambda \in \mathbb{R}^q\Big\},\\
\mathcal{N}_{C_2}(x_0, y_0)&=&\Big\{\displaystyle \sum_{i \in I(x_0, y_0)}\; \mu_i \partial g_i(x_0, y_0)\;:\;\mu_i \geq 0  \Big\}.
\end{array}
$$
This completes the announced characterization \eqref{caractsubQ}  of $\partial \mathcal{Q}( x_0)$. If $f$  and $g$ are differentiable then the condition \eqref{caractsubQ} can be written
\begin{equation}\label{caractsub}
\begin{array}{l}
s =  \nabla_x f(x_0, y_0)+ A^\transp \lambda + \displaystyle \sum_{i \in I(x_0, y_0)}\; \mu_i \nabla_x g_i(x_0, y_0),\;(a)\\
-\left[\nabla_y f(x_0, y_0)+ B^\transp \lambda + \displaystyle \sum_{i \in I(x_0, y_0)}\; \mu_i \nabla_y g_i(x_0, y_0)\right] \in \mathcal{N}_{Y}(y_0),\;(b)
\end{array}
\end{equation}
for some $\lambda \in \mathbb{R}^q$ and $\mu \in \mathbb{R}_{+}^{|I(x_0, y_0)|}$.

Finally, note that a primal-dual solution $(y_0, \lambda, \mu)$ satisfies \eqref{caractsub}-(b)
and if $(y_0, \lambda, \mu)$ with $\mu \geq 0$ satisfies \eqref{caractsub}-(b), knowing that $y_0$ is primal feasible, then under our assumptions $(\lambda, \mu)$
is a dual solution, i.e., $(\lambda, \mu) \in \Lambda(x_0)$.\hfill
\end{proof}
The following proposition provides conditions ensuring the 
Lipschitz continuity of $\mathcal{Q}$ and the boundedness
of its subdifferential at any point in $X$:
\begin{prop}\label{continuityvalf}
Consider the value function $\mathcal{Q}$ given by \eqref{vfunctionq}.
Let Assumption (H) hold and assume that 
for every $x \in X^{\varepsilon}$, the set $S(x)$ is nonempty, where $\varepsilon$
is given in (H)-3).
Then $\mathcal{Q}$ is finite on $X^\varepsilon$,
Lipschitz continuous on $X$, and the set $\cup_{x \in X} \partial \mathcal{Q}(x)$ is bounded.
More precisely, if $M_0=\sup_{x \in X^\varepsilon} \mathcal{Q}(x)$ and $m_0=\min_{x \in X} \mathcal{Q}(x)$, then for every $x \in X$ and every 
$s \in \partial \mathcal{Q}(x)$ we have
\begin{equation} \label{uppbounds}
\|s\|\leq \frac{1}{\varepsilon}(M_0 -m_0 ).
\end{equation}
\end{prop}
\begin{proof} 
Finiteness of $\mathcal{Q}$ on $X^\varepsilon$ follows from the fact that, under the assumptions of the lemma, for every $x \in X^\varepsilon$, the 
feasible set $S(x)$ of \eqref{vfunctionq} is nonempty and compact and the objective function $f(x, \cdot)$ is finite valued
on $Y$ and lower semicontinuous.
It follows that $X$ is contained in the relative interior of the domain of 
$\mathcal{Q}$. Since $\mathcal{Q}$ is convex and since a convex function is Lipschitz continuous
on the relative interior of its domain, $\mathcal{Q}$ is Lipschitz continuous on $X$. 

Next, for every $x \in X$, for every
$y \in X^\varepsilon$ and $s \in \partial \mathcal{Q}(x)$, we have
$$
\mathcal{Q}(y) \geq \mathcal{Q}(x) + \langle s, y-x \rangle.
$$
Observing  that $M_0$ and $m_0$ are finite ($\mathcal{Q}$ is finite and lower semicontinuous on the compact set $X^\varepsilon$),
for every $x \in X$ and $y \in X^\varepsilon$ we get
$$
M_0 \geq m_0 + \langle s, y-x \rangle.
$$
If $s = 0$ then \eqref{uppbounds} holds and 
if $s \neq 0$, taking $y=x+ \varepsilon \frac{s}{\|s\|} \in X^{\varepsilon}$  in the above relation, we obtain 
\eqref{uppbounds}, i.e., $s$ is bounded.\hfill
\end{proof}

\if{
\begin{rem} If $\emph{Aff}(X) \neq \mathbb{R}^m$, the proof of \eqref{continuityvalf} does not work if instead of (H)-3),
we use the weaker assumption:
$$
\mbox{there exists }\varepsilon>0 \mbox{ such that }(X^{\varepsilon} \cap \emph{Aff}(X)) {\small{\times}}  Y \subset \mbox{dom}(f).
$$
Indeed, assuming that $0 \in X$, then for any subgradient $s(x) \in \partial Q(x)$, the orthogonal projection $\Pi_{\emph{Aff}(X)}[s(x)]$  of 
$s(x)$ onto \emph{Aff}($X$) (which is a subspace since $0 \in X$) is still a subgradient of $\mathcal{Q}$ at $x$.
However, an arbitrary subgradient $s(x) \in \partial Q(x)$ does not necessarily
belong to \emph{Aff}($X$) since $\Pi_{\emph{Aff}(X)}[s(x)] +  y \in \partial Q(x)$ for any $y \in \emph{Aff}(X)^{\perp}$
and $\Pi_{\emph{Aff}(X)}[s(x)] +  y \notin \emph{Aff}(X)$. For this reason, in the next section,
we cannot replace Assumption (H2)-4) by Assumption (H1)-6) from \cite{lecphilgirar12}.
\end{rem}
}\fi

\section{Decomposition methods for risk-averse multistage stochastic convex programs}\label{framework}

Consider a risk-averse multistage stochastic optimization problem of the form
\begin{equation} \label{pbinit0}
\begin{array}{ll}
\displaystyle \inf_{x_1 \in X_1(x_0, \xi_{1})}&f_{1}(x_{1},\,\Psi_1) +  \rho_{2|\mathcal{F}_1}\left( \displaystyle \inf_{x_2 \in X_2(x_{0:1}, \xi_{2})}\;f_{2}(x_{1:2}, \Psi_2 ) + \ldots \right.\\
&+\rho_{T-1|\mathcal{F}_{T-2}}\left( \displaystyle \inf_{x_{T-1} \in X_{T-1}(x_{0:T-2},\,\xi_{T-1})}\;f_{T-1}(x_{1:T-1}, \Psi_{T-1}) \right.\\
&+ \rho_{T|\mathcal{F}_{T-1}}\left. \left. \left( \displaystyle \inf_{x_{T} \in X_{T}(x_{0:T-1},\,\xi_{T})}\;f_{T}(x_{1:T}, \Psi_{T})  \right) \right) \ldots \right)
\end{array}
\end{equation}
for some functions $f_t$ taking values in $\mathbb{R}\cup \{+\infty\}$, 
where 
$$
X_{t}(x_{0:t-1},\, \xi_{t}) = \Big\{x_t \in \mathcal{X}_t \;:\;g_t(x_{0:t}, \Psi_t) \leq 0,\;\;\displaystyle \sum_{\tau=0}^{t} \; A_{t, \tau} x_{\tau} = b_t\Big\}
$$
for some vector-valued functions $g_t$,
some random vectors $\Psi_t$ and $b_t$, some random matrices $A_{t, \tau}$, and where
$\xi_t$ is a discrete random vector with finite support corresponding to the concatenation of the
random variables $(\Psi_t, b_t,$ $(A_{t, \tau})_{\tau=0,\ldots,t})$
in an arbitrary order. In this problem $x_0$ is given, $\xi_1$ is deterministic, $(\xi_t)$ is 
a stochastic process, and setting $\mathcal{F}_{t}=\sigma(\xi_1,\ldots,\xi_t)$
and denoting by $\mathcal{Z}_t$ the set of $\mathcal{F}_t$-measurable functions, 
$\rho_{t+1|\mathcal{F}_t}: \mathcal{Z}_{t+1} \rightarrow \mathcal{Z}_{t}$ is a coherent and law invariant conditional risk measure.

In this section and the next two Sections \ref{sec:relativelcr} and \ref{sec:withoutrelativelcr}, we assume that the stochastic process $(\xi_t)$ 
satisfies the following assumption:\\
\begin{itemize}
\item[(H1)] $(\xi_t)$ 
is interstage independent and
for $t=2,\ldots,T$, $\xi_t$ is a random vector taking values in $\mathbb{R}^K$ with discrete distribution and
finite support $\{\xi_{t, 1}, \ldots, \xi_{t, M}\}$ while $\xi_1$ is deterministic 
($\xi_{t, j}$ is the vector corresponding to the concatenation of the elements in 
$(\Psi_{t, j}, b_{t, j}, (A_{t, \tau, j})_{\tau=0,\ldots,t}))$.\\
\end{itemize}
Under Assumption (H1), $\rho_{t+1|\mathcal{F}_t}$ coincides with its unconditional counterpart $\rho_{t+1}: \mathcal{Z}_{t+1} \rightarrow \mathbb{R}$.
To alleviate notation and without loss of generality, we assume that the number $M$ of possible realizations
of $\xi_t$, the size $K$ of $\xi_t$, and $n$ of $x_t$ do not depend on $t$.

For problem \eqref{pbinit0}, we can write the following dynamic programming equations:
we set $\mathcal{Q}_{T+1} \equiv 0$ and for
$t=2,\ldots,T$, define 
\begin{equation}\label{definitionQt}
\mathcal{Q}_{t}(x_{1:t-1})=\rho_t\Big(\mathfrak{Q}_{t}(x_{1:t-1}, \xi_t)\Big)
\end{equation}
with
\begin{equation} \label{defmathfrak}
\begin{array}{lll}
\mathfrak{Q}_{t}(x_{1:t-1}, \xi_t)&=&
\left\{
\begin{array}{l}
\displaystyle \inf_{x_t}\;F_t(x_{1:t}, \Psi_t):=f_t(x_{1:t}, \Psi_{t}) + \mathcal{Q}_{t+1}(x_{1:t})\\
x_t \in \mathcal{X}_t,\;g_t(x_{0:t}, \Psi_t) \leq 0,\;\;\displaystyle \sum_{\tau=0}^{t} \; A_{t, \tau} x_{\tau} = b_t,
\end{array}
\right.
\\
&=&\left\{
\begin{array}{l}
\displaystyle \inf_{x_t} \;F_t(x_{1:t}, \Psi_t)\\
x_t \in X_t (x_{0:t-1}, \xi_t ).
\end{array}
\right.
\end{array}
\end{equation}
With this notation, $F_t(x_{1:t}, \Psi_t)$ is the future optimal cost starting at time $t$
from the history of decisions $x_{1:t-1}$ if  $\Psi_t$ and $x_t$ are respectively the value of the
process $(\Psi_t)$ and the decision taken at stage $t$.  Problem \eqref{pbinit0} can then be written
\begin{equation} \label{firststagepb}
\left\{
\begin{array}{l}
\displaystyle \inf_{x_1} \; F_1(x_{1}, \Psi_1):=f_1(x_{1}, \Psi_{1}) + \mathcal{Q}_{2}(x_{1})\\
x_1 \in X_1 (x_{0}, \xi_1 )=\{
x_1 \in \mathcal{X}_1 : g_1(x_0, x_1, \Psi_1) \leq 0, A_{1,1} x_1 = b_1 - A_{1,0} x_0\},
\end{array}
\right.
\end{equation}
with optimal value denoted by $\mathcal{Q}_{1}(x_{0})=\mathfrak{Q}_{1}(x_{0}, \xi_1)$.
\par Setting $\Phi_{t, j} = \mathbb{P}(\xi_t = \xi_{t, j})>0$ for $j=1,\ldots,M$,   
we reformulate the problem as in \cite{philpmatos2} using the dual representation of a coherent risk measure \cite{artzner}:
\begin{equation} \label{defqtcoherent}
\mathcal{Q}_{t}(x_{1:t-1})=\rho_t( \mathfrak{Q}_{t}(x_{1:t-1}, \xi_t) ) = \displaystyle \sup_{p \in \mathcal{P}_t} \;\displaystyle \sum_{j=1}^{M} p_j \Phi_{t, j} \mathfrak{Q}_{t}(x_{1:t-1}, \xi_{t, j})
\end{equation}
for some convex subset $\mathcal{P}_t$
of 
$$
\mathcal{D}_t=\{p \in \mathbb{R}^M : p \geq 0,\;\;\sum_{j=1}^{M}\;p_j \Phi_{t, j} =1  \}.
$$ 
Optimization problem \eqref{defqtcoherent} is convex and linear if $\mathcal{P}_t$ is a polyhedron.
Such is the case when $\rho_t = CVaR_{1-\varepsilon_t}$ is the Conditional Value-at-Risk   of level $1-\varepsilon_t$
(introduced in \cite{ury2}) where (see \cite{philpmatos2} for instance)
$$
\mathcal{P}_t=\{p \in \mathcal{D}_t : p_j \leq \frac{1}{\varepsilon_t}, j=1,\ldots,M\}.
$$
In this case, the optimization problem \eqref{defqtcoherent} can be solved analytically, without resorting to an optimization step (once the values 
$\mathfrak{Q}_{t}(x_{1:t-1}, \xi_{t, j}), j=1,\ldots,M,$ are known, see \cite{philpmatos2} for details) and 
numerical simulations in Section 4.1.1 of \cite{shaptekaya2}
have shown that the corresponding subproblems are solved more quickly than if the minimization formula from \cite{ury1}, \cite{ury2} for the Conditional Value-at-Risk
was used. We refer to \cite{ruzshap2}, \cite{ruzshap1}, \cite{rom2}, \cite{guiguesrom10}
for the definition of the sets $\mathcal{P}_t$ corresponding to various popular risk measures.

Recalling definition \eqref{epsfatten} of the
the $\varepsilon$-fattening of a set $X$, we also make the following Assumption (H2) for $t=1,\ldots,T$:
\begin{itemize}
\item[1)] $\mathcal{X}_{t} \subset \mathbb{R}^n$ is nonempty, convex, and compact.
\item[2)] For every $x_{1:t} \in \mathbb{R}^{n} \times \ldots  \times \mathbb{R}^{n}$ the function
$f_t(x_{1:t}, \cdot)$ is measurable and for every $j=1,\ldots,M$, the function
$f_t(\cdot, \Psi_{t, j})$ is proper, convex, and lower semicontinuous.
\item[3)] For every $j=1,\ldots,M$, each component of the function $g_t(x_0, \cdot, \Psi_{t, j})$ is a
convex lower semicontinuous function.
\item[4)] There exists $\varepsilon>0$ such that:
\begin{itemize}
\item[4.1)] for every $j=1,\ldots,M$,
$$
\Big[\mathcal{X}_1 {\small{\times}} \ldots {\small{\times}} \mathcal{X}_{t-1} \Big]^{\varepsilon}{\small{\times}} \mathcal{X}_{t}  \subset \mbox{dom} \; f_t\Big(\cdot, \Psi_{t, j}\Big);
$$
\item[4.2)] for every $j=1,\ldots,M$, for every
$x_{1:t-1} \in \Big[\mathcal{X}_1 {\small{\times}} \ldots {\small{\times}} \mathcal{X}_{t-1}\Big]^{\varepsilon}$,
the set $X_t(x_{0:t-1}, \xi_{t, j})$ is nonempty.
\end{itemize}
\item[5)] If $t \geq 2$, for every $j=1,\ldots,M$, there exists
$$
{\bar x}_{t, j}=({\bar x}_{t, j, 1}, \ldots, {\bar x}_{t, j, t}) \in \mathcal{X}_1 \small{\times}\ldots \small{\times}\mathcal{X}_{t-1}\small{\times}\mbox{ri}(\mathcal{X}_t)
\cap \mbox{ri}(\{g_t(x_0,\cdot,\Psi_{t, j})\leq 0\})$$ such that $\bar x_{t, j, t} \in X_t(x_0, \bar x_{t, j, 1},\ldots,\bar x_{t, j, t-1}, \xi_{t, j})$.\\
\end{itemize}

As shown in Proposition \ref{convexityrec}, Assumption (H2) guarantees that 
for $t=2,\ldots,T$, recourse function $\mathcal{Q}_t$ is convex and Lipschitz  continuous on
the set $\Big[\mathcal{X}_1{\small{\times}} \ldots {\small{\times}} \mathcal{X}_{t-1}\Big]^{\hat{\varepsilon}}$
for every $0<{\hat{\varepsilon}}<\varepsilon$.
\begin{prop}\label{convexityrec} 
Under Assumption (H2), for $t=2,\ldots,T+1$, for every $0<{\hat{\varepsilon}}<\varepsilon$, the recourse function $\mathcal{Q}_t$ is convex,
finite on $\Big[\mathcal{X}_1 \small{\times} \ldots \small{\times}\mathcal{X}_{t-1}\Big]^{\hat{\varepsilon}}$,
and continuous on
$\Big[\mathcal{X}_1{\small{\times}} \ldots {\small{\times}} \mathcal{X}_{t-1}\Big]^{\hat{\varepsilon}}$.
\end{prop}
\begin{proof} The proof is by induction on $t$.
The result holds for $t=T+1$ since $\mathcal{Q}_{T+1} \equiv 0$. Now assume that for
some $t \in \{2,\ldots,T\}$, the function $\mathcal{Q}_{t+1}$
is convex, 
finite on
$\Big[\mathcal{X}_1 \small{\times} \ldots \small{\times}\mathcal{X}_{t}\Big]^{\hat{\varepsilon}}$,
and continuous on
$\Big[\mathcal{X}_1{\small{\times}} \ldots {\small{\times}} \mathcal{X}_{t}\Big]^{\hat{\varepsilon}}$
for every $0<{\hat{\varepsilon}}<\varepsilon$.
Take an arbitrary 
$0<{\hat{\varepsilon}}<\varepsilon$,
$x_{1:t-1} \in \Big[\mathcal{X}_1 {\small{\times}} \ldots {\small{\times}} \mathcal{X}_{t-1} \Big]^{{\hat{\varepsilon}}}$
and fix $j \in \{1,\ldots,M\}$.
Consider the optimization problem \eqref{defmathfrak} with $\xi_t=\xi_{t, j}$.
Note that the feasible set $X_t(x_{0:t-1}, \xi_{t, j})$ of this problem 
is nonempty (invoking (H2)-4.2)) and compact, since it is the intersection of the  compact set $\mathcal{X}_t$
(invoking (H2)-1)), an affine space, and a lower level set of $g_t(x_{0:t-1}, \cdot, \Psi_{t, j})$ which is closed since
this function $g_t(x_{0:t-1}, \cdot, \Psi_{t, j})$  is lower semicontinuous (using Assumption (H2)-3)).
Next observe that if $x_{1:t-1} \in \Big[\mathcal{X}_1 \small{\times} \ldots \small{\times}\mathcal{X}_{t-1}\Big]^{\hat{\varepsilon}}$
and $x_t \in \mathcal{X}_t^{\tilde \varepsilon}$ with ${\tilde \varepsilon}=\sqrt{\left( \frac{\varepsilon + {\hat \varepsilon}}{2} \right)^2 - {\hat \varepsilon}^2}>0$,
then $x_{1:t} \in \Big[\mathcal{X}_1 {\small{\times}} \ldots {\small{\times}} \mathcal{X}_{t} \Big]^{\frac{\varepsilon + {\hat \varepsilon}}{2}}$ with $(\varepsilon + {\hat \varepsilon}   )/2<\varepsilon$. 
Using this observation and the induction hypothesis, we have that $\mathcal{Q}_{t+1}(x_{1:t-1}, \cdot)$ is finite (and convex) on $\mathcal{X}_t^{\tilde \varepsilon}$
which implies that $\mathcal{Q}_{t+1}(x_{1:t-1}, \cdot)$ is Lipschitz continuous on $\mathcal{X}_t$.
It follows that the optimal value $\mathfrak{Q}_{t}(x_{1:t-1}, \xi_{t, j})$ of problem \eqref{defmathfrak} with $\xi_t=\xi_{t, j}$ is finite
because its objective function $x_t \rightarrow f_t(x_{1:t-1}, x_t,\Psi_{t, j})+\mathcal{Q}_{t+1}(x_{1:t-1}, x_t)$ takes finite values
on $\mathcal{X}_t$ (using (H2)-4.1), (H2)-2),  and the induction hypothesis) and is lower semicontinuous (using
(H2)-2)). Using Definition \eqref{defqtcoherent} of $\mathcal{Q}_{t}$, we deduce that $\mathcal{Q}_{t}(x_{1:t-1})$ is finite.
Since  $x_{1:t-1}$
was chosen arbitrarily
in $\Big[\mathcal{X}_1 {\small{\times}} \ldots {\small{\times}} \mathcal{X}_{t-1} \Big]^{\hat \varepsilon}$,
we have shown that $\mathcal{Q}_t$ is finite on 
$\Big[\mathcal{X}_1 {\small{\times}} \ldots {\small{\times}} \mathcal{X}_{t-1} \Big]^{\hat \varepsilon}$.

Next, we deduce from Assumptions (H2)-1), (H2)-2), and (H2)-3) that for every $j\in \{1,\ldots,M\}$, $\mathfrak{Q}_t(\cdot, \xi_{t, j})$ is convex
on $\Big[\mathcal{X}_1 {\small{\times}} \ldots {\small{\times}} \mathcal{X}_{t-1} \Big]^{\hat \varepsilon}$.
Since $\rho_t$ is coherent, it is monotone and convex, and 
$\mathcal{Q}_t(\cdot)=\rho_t(\mathfrak{Q}_t(\cdot, \xi_{t}))$ is convex
on $\Big[\mathcal{X}_1 {\small{\times}} \ldots {\small{\times}} \mathcal{X}_{t-1} \Big]^{\hat \varepsilon}$.
Since $\Big[\mathcal{X}_1 {\small{\times}}\ldots {\small{\times}} \mathcal{X}_{t-1}\Big]^{\hat \varepsilon}$ is a compact subset of the relative
interior of the domain of convex function $\mathcal{Q}_t$, we have that $\mathcal{Q}_t$ is Lipschitz continuous on
$\Big[\mathcal{X}_1 {\small{\times}}\ldots {\small{\times}} \mathcal{X}_{t-1}\Big]^{\hat \varepsilon}$.\hfill
\end{proof}

Recalling Assumption (H1), the distribution of $(\xi_2, \ldots, \xi_T)$ is discrete 
and the $M^{T-1}$ possible realizations of $(\xi_2, \ldots, \xi_T)$ can be organized
in a finite tree with the root node $n_0$ associated to a stage $0$ (with decision $x_0$ taken at that
node) having one child node $n_1$
associated to the first stage (with $\xi_1$ deterministic).
Algorithm 1 below is a sampling algorithm which, for
iteration $k \geq 1$, selects a set of nodes $(n_1^k, n_2^k, \ldots, n_T^k)$ of the scenario tree
(with $n_t^k$ a node of stage $t$) 
corresponding to a sample $({\tilde \xi}_1^k, {\tilde \xi}_2^k,\ldots, {\tilde \xi}_T^k)$
of $(\xi_1, \xi_2,\ldots, \xi_T)$. Since ${\tilde \xi}_1^k=\xi_1$, we have $n_1^k=n_1$ for all $k$.

In the sequel, we use the following notation:
$\mathcal{N}$ is the set of nodes and 
$\mathcal{P}:\mathcal{N} \rightarrow \mathcal{N}$ is the function 
associating to a node its parent node (the empty set for the root node).
We will denote by {\tt{Nodes}}$(t)$ the set of nodes for stage $t$ and
for a node $n$ of the tree, we denote by 
\begin{itemize}
\item $C(n)$ the set of its children nodes (the empty set for the leaves);
\item $x_n$ a decision taken at that node;
\item $\Phi_n$ the transition probability from the parent node of $n$ to $n$;
\item $\xi_n$ the realization of process $(\xi_t)$ at node $n$\footnote{Note that
to alleviate notation, the same notation $\xi_{\tt{Index}}$ is used to denote
the realization of the process at node {\tt{Index}} of the scenario tree and the value of the process $(\xi_t)$
for stage {\tt{Index}}. The context will allow us to know which concept is being referred to.
In particular, letters $n$ and $m$ will only be used to refer to nodes while $t$ will be used to refer to stages.}:
for a node $n$ of stage $t$, this realization $\xi_n$ is the concatenation of the realizations $\Psi_n$ of $\Psi_t$,
$b_n$ of $b_t$, and $A_{\tau, n}$ of $A_{t, \tau}$ for $\tau=0,1,\ldots,t$;
\item $\xi_{[n]}$ (resp. $x_{[n]}$) the history of the realizations of the process $(\xi_t)$ (resp. the history of the decisions) from the first stage node $n_1$ to node $n$:
for a node $n$ of stage $t$, the $i$-th component of $\xi_{[n]}$ (resp. $x_{[n]}$) is $\xi_{\mathcal{P}^{t-i}(n)}$ (resp. $x_{\mathcal{P}^{t-i}(n)}$) for $i=1,\ldots, t$;
\end{itemize}

We are now in a position to describe Algorithm 1 which is a decomposition algorithm
solving \eqref{pbinit0}. This algorithm exploits the convexity of recourse functions $\mathcal{Q}_{t},\;t=2,\ldots,T+1$,
building polyhedral lower approximations $\mathcal{Q}_{t}^{k},\;t=2,\ldots,T+1$, of these functions of the form
$$
\begin{array}{lll}
\mathcal{Q}_{t}^{k}(x_{1:t-1}) & = &\displaystyle \max_{0 \leq \ell \leq k}\;\Big( \theta_t^{\ell} + \langle \beta_t^{\ell}, x_{1:t-1}-x_{[n_{t-1}^{\ell}]}^{\ell} \rangle \Big)\\
& = &\displaystyle \max_{0 \leq \ell \leq k}\;\Big( \theta_t^{\ell} + \langle \beta_{t, 1}^{\ell}, x_{1:t-2}-x_{[n_{t-2}^{\ell}]}^{\ell} \rangle + \langle \beta_{t, 2}^{\ell}, x_{t-1}-x_{n_{t-1}^{\ell}}^{\ell} \rangle \Big),
\end{array}
$$ 
where $\beta_{t, 1}^{\ell} \in \mathbb{R}^{n(t-2)}$ (resp. $\beta_{t, 2}^{\ell} \in \mathbb{R}^{n}$) gathers the first
$n(t-2)$ (resp. last $n$) components of $\beta_t^{\ell}$.

Since $\mathcal{Q}_{T+1}\equiv 0$ is known, we have $\mathcal{Q}_{T+1}^k\equiv 0$ for all $k$, i.e.,
$\theta_{T+1}^k$ and $\beta_{T+1}^k$ are null for all $k \in \mathbb{N}$.
At iteration $k$, decisions $(x_{n_1^k}^k, \ldots,x_{n_T^k}^k)$ and
coefficients $(\theta_t^k, \beta_t^k),\;t=2,\ldots,T+1$, are computed
for a sample of nodes $(n_1^k, n_2^k, \ldots, n_T^k)$: $x_{n_t^k}^k$ is the decision taken at node $n_t^k$
replacing the (unknown) recourse function $\mathcal{Q}_{t+1}$ by 
$\mathcal{Q}_{t+1}^{k-1}$, available at the beginning of iteration $k$.

In Lemma \ref{lipqtk} below, we show that the coefficients $(\theta_t^k, \beta_t^k)$
computed in Algorithm 1 define valid cuts for $\mathcal{Q}_t$, i.e., $\mathcal{Q}_t \geq \mathcal{Q}_t^k$
for all $k \in \mathbb{N}$. To describe Algorithm 1, it is convenient to introduce for $t=2,\ldots, T$, the function
$\mathfrak{Q}_{t}^{k-1}$ defined as follows: $\mathfrak{Q}_{t}^{k-1}(x_{1:t-1}, \xi_t)$ is the optimal value of the optimization problem
\begin{equation} \label{eqdefapprox}
\left\{
\begin{array}{l}
\displaystyle \inf_{x_t}\;F_t^{k-1}(x_{1:t}, \Psi_t):=f_t(x_{1:t}, \Psi_{t}) + \mathcal{Q}_{t+1}^{k-1}(x_{1:t})\\
x_t \in \mathcal{X}_t, g_t(x_{0:t}, \Psi_t) \leq 0,\;\;\displaystyle  \sum_{\tau=0}^{t} \; A_{t, \tau} x_{\tau} = b_t.
\end{array}
\right.
\end{equation}
We also denote by $\mathfrak{Q}_{1}^{k-1}(x_{0}, \xi_1)$ the optimal value of the problem above for $t=1$.\\
\rule{\linewidth}{1pt}
{\textbf{Algorithm 1: Multistage stochastic decomposition algorithm to solve \eqref{pbinit0}.}}\\\newline
{\textbf{Initialization.}} Set $\mathcal{Q}_{t}^0 \equiv-\infty$ for $t=2,\ldots,T$, and
$\mathcal{Q}_{T+1}^0 \equiv 0$, i.e., set $\theta_{T+1}^0=0$, $\beta_{t+1}^0=0$  for $t=1,\ldots,T$, 
and $\theta_{t+1}^0=-\infty$ for $t=1,\ldots,T-1$.\\\newline
{\textbf{Loop.}}\\
{\textbf{For $k=1,2,\ldots,$}}\\
\hspace*{0.4cm}Sample a set of $T+1$ nodes $(n_0^k, n_1^k, \ldots, n_T^k)$ such that
$n_0^k$ is the root node, $n_1^k = n_1$ is the node\\
\hspace*{0.4cm}corresponding to the first stage, and for every
$t=2,\ldots,T$, node $n_t^k$ is a child node of node $n_{t-1}^k$.\\
\hspace*{0.4cm}This set of nodes is associated to a sample $({\tilde \xi}_1^k, {\tilde \xi}_2^k,\ldots, {\tilde \xi}_T^k)$
of $(\xi_1, \xi_2,\ldots, \xi_T)$, realizations of random\\
\hspace*{0.4cm}variables $(\xi_1^k, \xi_2^k,\ldots, \xi_T^k )$. \\
\hspace*{0.4cm}{\textbf{For}} $t=1,\ldots,T$,\\
\hspace*{0.8cm}{\textbf{For}} every node $n$ of stage $t-1$\\
\hspace*{1.2cm}{\textbf{For}} every child node $m$ of node $n$, compute an optimal solution
$x_m^k$ of \eqref{eqdefapprox} taking $(x_{0:t-1} , \xi_t)=$
\hspace*{1.6cm}$(x_0, x_{[n]}^k, \xi_m)$ solving 
\begin{equation} \label{defxtkj}
\begin{array}{l}
\left\{
\begin{array}{l}
\displaystyle \inf_{x_t} \;f_t(x_{[n]}^k, x_t, \Psi_{m}) + \mathcal{Q}_{t+1}^{k-1}(x_{[n]}^k, x_t)\\
g_t(x_0, x_{[n]}^k, x_t, \Psi_{m}) \leq 0,\\
A_{t, m} x_{t} = b_{m}-\displaystyle \sum_{\tau=0}^{t-1} \; A_{\tau, m} x_{\mathcal{P}^{t-\tau} (m)}^k,\\
x_t \in \mathcal{X}_t,
\end{array}
\right.
=\left\{
\begin{array}{l}
\displaystyle \inf_{x_t, z} \;f_t(x_{[n]}^k, x_{t}, \Psi_{m}) + z\\
g_t(x_0, x_{[n]}^k, x_t, \Psi_{m}) \leq 0,\;\;\;\hspace*{2.1cm}[\pi_{k, m, 1}]\\ 
A_{t, m} x_{t} = b_{m}-\displaystyle \sum_{\tau=0}^{t-1} \; A_{\tau, m} x_{\mathcal{P}^{t-\tau} (m)}^k,\;\;\;\hspace*{0.39cm}[\pi_{k, m, 2}]\\
z \geq \theta_{t+1}^{\ell} + \langle \beta_{t+1, 1}^{\ell}, x_{[n]}^{k}-x_{[n_{t-1}^{\ell}]}^{\ell} \rangle \\
\;\;\;\;\;\;+ \langle \beta_{t+1, 2}^{\ell}, x_{t}-x_{n_t^{\ell}}^{\ell} \rangle,0 \leq \ell \leq k-1,\;\hspace*{0.05cm}[\pi_{k, m, 3}]\\
x_t \in \mathcal{X}_t,
\end{array}
\right.
\end{array}
\end{equation}
\hspace*{1.6cm}with optimal value $\mathfrak{Q}_t^{k-1}(x_{[n]}^{k}, \xi_{m})$ for $t \geq 2$, $\mathfrak{Q}_1^{k-1}(x_0, \xi_{1})$ for $t=1$, and 
with the convention,
\hspace*{1.6cm}for $t=1$, that $(x_0, x_{[n_0^k]}^k, x_1)=(x_0, x_1)$ and  $x_{[n_{0}^{k}]}^k = x_{n_0^k}^k=x_{n_0}^k=x_0$ for all $k$. \\
\hspace*{1.6cm}In the above problem, we have denoted by $\pi_{k, m, 1}, \pi_{k, m, 2}$ and $\pi_{k, m, 3}$ the optimal Lagrange\\
\hspace*{1.6cm}multipliers associated with respectively the first, second, and third group of constraints.\\
\hspace*{1.2cm}{\textbf{End For}}\\
\hspace*{1.6cm}{\textbf{If}} $n=n_{t-1}^k$ and $t \geq 2$, compute for every $m \in C(n)$
$$
\begin{array}{lll}
\pi_{k , m}& = &f'_{t, x_{1:t-1}}\Big(x_{[n]}^k, x_{m}^k, \Psi_{m} \Big)+ g'_{t, x_{1:t-1}}\Big(x_0, x_{[n]}^k, x_m^k, \Psi_{m} \Big) \pi_{k , m, 1}\\
&&+ \left(\begin{array}{c}A_{1, m}^\transp\\\vdots\\A_{t-1, m}^\transp \end{array}   \right) \pi_{k , m, 2}+ [\beta_{t+1, 1}^0, \beta_{t+1, 1}^{1}, \ldots, \beta_{t+1, 1}^{k}]\pi_{k , m, 3}
\end{array}
$$
\hspace*{1.6cm}where $f'_{t, x_{1:t-1}}(x_{[n]}^k, x_{m}^k, \Psi_{m})$ is a 
subgradient of convex function $f_t ( \cdot, x_{m}^k, \Psi_{m} )$ at $x_{[n]}^k$ and
the $i$-th 
\hspace*{1.6cm}column of matrix $g'_{t, x_{1:t-1}}(x_0, x_{[n]}^k, x_{m}^k, \Psi_{m})$ is a subgradient at $x_{[n]}^k$ of the $i$-th component\\
\hspace*{1.6cm}of convex function $g_{t}(x_0, \cdot, x_{m}^k, \Psi_{m})$.\\
\hspace*{1.6cm}Compute $p_{k, m},  m \in C(n)$, solving 
$$
\begin{array}{lll}
\rho_t \left( \mathfrak{Q}_{t}^{k-1}(x_{[n]}^{k}, \xi_t )  \right)&= &\displaystyle \sup_{p \in \mathcal{P}_t} \;\displaystyle \sum_{m \in C(n)} p_m \Phi_{m} \mathfrak{Q}_{t}^{k-1}(x_{[n]}^k, \xi_{m})\\
&=&\displaystyle \sum_{m \in C(n)} p_{k , m} \Phi_{m} \mathfrak{Q}_{t}^{k-1}(x_{[n]}^{k}, \xi_{m}).
\end{array}
$$
\hspace*{1.6cm}Compute coefficients
\begin{equation}\label{formulathetak}
\begin{array}{l}
\theta_t^k  = \rho_t \left( \mathfrak{Q}_{t}^{k-1}(x_{[n]}^{k}, \xi_t) \right)=  \displaystyle \sum_{m \in C(n)} p_{k , m} \Phi_{m} \mathfrak{Q}_t^{k-1}(x_{[n]}^k, \xi_{m}) \mbox{ and }
\beta_t^k  = \displaystyle \sum_{m \in C(n)} p_{k , m} \Phi_{m}  \pi_{k , m}, 
\end{array}
\end{equation}
\hspace*{1.6cm}making up the new approximate recourse function
$$
\mathcal{Q}_{t}^{k}(x_{1:t-1})  = \displaystyle \max_{0 \leq \ell \leq k}\;\Big( \theta_t^{\ell} + \langle \beta_t^{\ell}, x_{1:t-1}-x_{[n_{t-1}^{\ell}]}^{\ell} \rangle \Big).
$$
\hspace*{1.6cm}{\textbf{End If}}\\
\hspace*{0.8cm}{\textbf{End For}}\\ 
\hspace*{0.4cm}{\textbf{End For}}\\
\hspace*{0.4cm}Compute $\theta_{T+1}^k=0$ and $\beta_{T+1}^k=0$.\\
{\textbf{End For}}\\
\rule{\linewidth}{1pt}

The convergence of Algorithm 1 is shown in the next section. Various modifications of Algorithm 1 have been proposed in the literature. For instance, it is possible to 
\begin{itemize}
\item[(i)] use a number of samples that varies along the iterations;
\item[(ii)] sample from the distribution of $\xi_t$ (instead of using all realizations $\xi_{t, 1}, \ldots, \xi_{t, M}$ of $\xi_t$)
to build the cuts \cite{chenpowell99}, \cite{philpot};
\item[(iii)] generate the trial points $x_{n_t^k}^k$ using the Abridged Nested Decomposition Method \cite{birgedono}.
\end{itemize}
The convergence proof of the next section can be extended to these variants of Algorithm 1.

We will assume that the sampling procedure in Algorithm 1 satisfies the following property:\\
\begin{itemize}
\item[(H3)] for every $j=1,\ldots,M$, for every $t=2,\ldots,T$, and for every $k \in \mathbb{N}^*$,
$
\mathbb{P}(\xi_t^k = \xi_{t, j}) = \Phi_{t, j}>0 \mbox{ with }\sum_{j=1}^{M} \Phi_{t, j}=1.
$ For every 
$t=2,\ldots,T$, and $k \geq 1$,
$$
\xi_t^k \mbox{ is independent on }\sigma(\xi_2^1,\ldots,\xi_T^1,\ldots,\xi_{2}^{k-1}, \ldots, \xi_{T}^{k-1},\xi_2^k,\ldots,\xi_{t-1}^k).
$$
\end{itemize}

In the following three lemmas, we show item (ii) announced in the introduction: functions $\mathcal{Q}_t^k$
for $k \geq T-t+1$ are Lipschitz continuous and have bounded subgradients. 
\begin{lemma}\label{lipqtk}
Consider the sequences $\mathcal{Q}_t^k, \theta_t^k$, and $\beta_t^k$ generated by Algorithm 1.
Under Assumptions (H2), then almost surely, for $t=2,\ldots,T+1$, the following holds:
\begin{itemize}
\item[(a)] $\mathcal{Q}_t^k$ is convex with
$\mathcal{Q}_t^{k} \leq \mathcal{Q}_{t}$ on 
$\Big[\mathcal{X}_1 \small{\times}\ldots \small{\times}\mathcal{X}_{t-1}\Big]^{\varepsilon}$ for all $k \geq 1$;
\item[(b)] the sequences
$(\theta_t^k)_{k \geq T-t+1}$, $(\beta_t^k)_{k \geq T-t+1}$, and
$(\pi_{k, m})_{k \geq T-t+1}$ for all $m$,
are bounded;
\item[(c)] for $k \geq T-t+1$,
$\mathcal{Q}_t^k$ is convex Lipschitz continuous on 
$\Big[\mathcal{X}_1 \small{\times} \ldots \small{\times}\mathcal{X}_{t-1}\Big]^{\varepsilon}$.
\end{itemize}
\end{lemma}
\begin{proof} 
We show the result by induction on $k$ and $t$.
For $t=T+1$, and $k \geq 0$, $\theta_{t}^k$ and 
$\beta_{t}^k$ are 
bounded since they are null (recall that $\mathcal{Q}_{T+1}^k$ is null for all $k \geq 0$) and
$\mathcal{Q}_{T+1}^k = \mathcal{Q}_{T+1} \equiv 0$ is convex and Lipschitz continuous on
$\mathcal{X}_1{\small{\times}}\ldots {\small{\times}} \mathcal{X}_{T}$
for $k \geq 0$.
Assume now that for some $t \in \{1, \ldots,T\}$ and $k \geq T-t+1$, the functions $\mathcal{Q}_{t+1}^j$
for $T-t \leq j \leq k-1$ are convex Lipschitz continuous on 
$\Big[\mathcal{X}_1 \small{\times} \ldots \small{\times}\mathcal{X}_{t}\Big]^{\varepsilon}$
with $\mathcal{Q}_{t+1}^j  \leq \mathcal{Q}_{t+1}$.
We show that (i) $\theta_t^k$ and $\beta_t^k$ are well defined and bounded;
(ii) $\mathcal{Q}_t^k$ is convex Lipschitz continuous on 
$\Big[\mathcal{X}_1 \small{\times} \ldots \small{\times}\mathcal{X}_{t-1}\Big]^{\varepsilon}$;
(iii) $\mathcal{Q}_t \geq \mathcal{Q}_t^k$ on 
$\Big[\mathcal{X}_1 \small{\times} \ldots \small{\times}\mathcal{X}_{t-1}\Big]^{\varepsilon}$.

Take an arbitrary $x_{1:t-1} \in \Big[\mathcal{X}_1 \small{\times} \ldots \small{\times}\mathcal{X}_{t-1}\Big]^{\varepsilon}$.
Since $\mathcal{Q}_{t+1} \geq \mathcal{Q}_{t+1}^{k-1}$ on
$\Big[\mathcal{X}_1 \small{\times} \ldots \small{\times}\mathcal{X}_{t}\Big]^{\varepsilon}$,
using definition \eqref{defmathfrak} of $\mathfrak{Q}_t$ and the definition of $\mathfrak{Q}_t^{k}$, we have 
\begin{equation} \label{ineqrfak}
\mathfrak{Q}_{t}(x_{1:t-1}, \cdot) \geq \mathfrak{Q}_{t}^{k-1}(x_{1:t-1}, \cdot)
\end{equation}
and using the monotonicity of $\rho_t$ (recall that $\rho_t$ is coherent)
\begin{equation}\label{inequalityqfrac}
\begin{array}{lll}
\mathcal{Q}_{t}(x_{1:t-1})=\rho_t\Big(  \mathfrak{Q}_{t}(x_{1:t-1}, \xi_t) \Big) & \geq & \rho_t \Big( \mathfrak{Q}_{t}^{k-1}(x_{1:t-1}, \xi_t) \Big)\\
& \geq & \displaystyle \sup_{p \in \mathcal{P}_t} \;\displaystyle \sum_{j=1}^{M} p_j \Phi_{t, j} \mathfrak{Q}_{t}^{k-1}(x_{1:t-1}, \xi_{t, j}).
\end{array}
\end{equation}
Using Assumptions (H2)-1), 2), 3), 4.1), 4.2)
and the fact that $\mathcal{Q}_{t+1}^{k-1}$ is Lipschitz continuous on the compact set 
$\Big[\mathcal{X}_1 \small{\times} \ldots \small{\times}\mathcal{X}_{t}\Big]^{\varepsilon}$
(induction hypothesis), we have that $\mathfrak{Q}_{t}^{k-1}(x_{1:t-1}, \xi_{t, j})$ is finite for all $j \in \{1,\ldots,M\}$.
Recalling that 
$$\theta_t^k  = \displaystyle \sum_{m \in C(n)} p_{k , m} \Phi_{m} \mathfrak{Q}_t^{k-1}(x_{[n]}^k, \xi_{m})$$
with $x_{[n]}^k \in \mathcal{X}_1 \small{\times} \ldots \small{\times}\mathcal{X}_{t-1} \subset \Big[\mathcal{X}_1 \small{\times} \ldots \small{\times}\mathcal{X}_{t-1}\Big]^{\varepsilon}$,
it follows that $\theta_t^k$
is finite.
Next, using Assumptions (H2)-2), 3), for every $j \in \{1,\ldots,M\}$, the function $\mathfrak{Q}_{t}^{k}(\cdot, \xi_{t, j})$ is convex.
Since it is finite on 
$\Big[\mathcal{X}_1 \small{\times} \ldots \small{\times}\mathcal{X}_{t-1}\Big]^{\varepsilon}$,
it is Lipschitz continuous on 
$\mathcal{X}_1 {\small{\times}} \ldots {\small{\times}} \mathcal{X}_{t-1}$.
This function is thus subdifferentiable on 
$\Big[\mathcal{X}_1 \small{\times} \ldots \small{\times}\mathcal{X}_{t-1}\Big]^{\varepsilon}$
and using Lemma \ref{dervaluefunction}, whose assumptions are satisfied,
$\pi_{k , m}$ is a subgradient of $\mathfrak{Q}_{t}^{k-1}(\cdot, \xi_{m})$
at $x_{[n_{t-1}^k]}^k$, i.e., setting $n=n_{t-1}^k$, for every $x_{1:t-1} \in \Big[\mathcal{X}_1 \small{\times} \ldots \small{\times}\mathcal{X}_{t-1}\Big]^{\varepsilon}$, 
we have
\begin{equation}\label{relasubgradqt}
\mathfrak{Q}_t^{k-1}(x_{1:t-1}, \xi_{m}) 
\geq \mathfrak{Q}_t^{k-1}(x_{[n]}^k, \xi_{m}) + \langle \pi_{k , m},  x_{1:t-1} - x_{[n]}^{k} \rangle.
\end{equation}
Plugging this inequality into \eqref{inequalityqfrac}, still denoting $n=n_{t-1}^k$,
we obtain for $x_{1:t-1} \in \Big[\mathcal{X}_1 \small{\times} \ldots \small{\times}\mathcal{X}_{t-1}\Big]^{\varepsilon}$:
$$
\begin{array}{lcl} 
\mathcal{Q}_{t}(x_{1:t-1}) & \geq  & \displaystyle \sup_{p \in \mathcal{P}_t} \;\displaystyle \sum_{j=1}^M p_j \Phi_{t, j} \mathfrak{Q}_{t}^{k-1}(x_{1:t-1}, \xi_{t, j})\\
&=  &  \displaystyle \sup_{p \in \mathcal{P}_t} \; \displaystyle \sum_{m \in C(n)} p_{m} \Phi_{m} \mathfrak{Q}_{t}^{k-1}(x_{1:t-1}, \xi_{m}) \\
& \geq  & \displaystyle \sum_{m \in C(n)} p_{k , m} \Phi_{m} \mathfrak{Q}_{t}^{k-1}(x_{1:t-1}, \xi_{m}) \mbox{ since }p_k=(p_{k, m})_{m \in C(n)} \in \mathcal{P}_t\\
& \stackrel{\eqref{relasubgradqt}}{\geq}  & \displaystyle \sum_{m \in C(n)} p_{k , m} \Phi_{m} \mathfrak{Q}_t^{k-1}(x_{[n]}^k, \xi_{m}) + \displaystyle \sum_{m \in C(n)} p_{k, m} \Phi_{m} \langle \pi_{k, m}, x_{1:t-1} - x_{[n]}^{k}\rangle\\
& \geq  & \theta_t^k + \langle \beta_t^k, x_{1:t-1}-x_{[n]}^k \rangle
\end{array}
$$
using the definitions of $\theta_t^k$ and  $\beta_t^k$.
If $\beta_t^k=0$ then $\beta_t^k$ is bounded and if $\beta_t^k \neq 0$, plugging
$x_{1:t-1}=x_{[n]}^k + \frac{\varepsilon}{2} \frac{\beta_t^k}{\|\beta_t^k\|}\in \Big[\mathcal{X}_1 \small{\times} \ldots \small{\times}\mathcal{X}_{t-1}\Big]^{\varepsilon/2}$
in the above inequality, where $\varepsilon$ is defined in (H2)-4.2), we obtain
\begin{equation} \label{uboundbeta}
\|\beta_t^k\| \leq \frac{2}{\varepsilon}\left(\mathcal{Q}_{t}\Big(x_{[n]}^k + \frac{\varepsilon}{2} \frac{\beta_t^k}{\|\beta_t^k\|}\Big)-\theta_t^k \right).
\end{equation}
From Proposition \ref{convexityrec}, $\mathcal{Q}_t$ is finite
on $\Big[\mathcal{X}_1 \small{\times} \ldots \small{\times}\mathcal{X}_{t-1}\Big]^{\varepsilon/2}$.
Since $\theta_t^k$ is finite, \eqref{uboundbeta} shows that $\beta_t^k$ is bounded:
\begin{equation} \label{boundbeta}
\|\beta_t^k\| \leq \frac{2}{\varepsilon}\left(\displaystyle \sup_{x_{1:t-1} \in [\mathcal{X}_1 \small{\times} \ldots \small{\times}\mathcal{X}_{t-1}]^{\varepsilon/2}} \mathcal{Q}_{t}\Big(x_{1:t-1}\Big)-\theta_t^k \right).
\end{equation}
This achieves the induction step.
Gathering our observations, we have shown that $\mathcal{Q}_t \geq \mathcal{Q}_t^k$ for all $k \in \mathbb{N}$ and that
$\mathcal{Q}_t^k$ is Lipschitz continuous for $k \geq T-t+1$.

Finally, using Proposition \ref{continuityvalf}, we have that $\pi_{k, m}$ is bounded.
More precisely,
If $\pi_{k, m} \neq 0$, then relation \eqref{relasubgradqt} written for 
$x_{1:t-1}={\tilde x}_{1:t-1}^{k, m}=x_{[n]}^k + \frac{\varepsilon}{2} \frac{\pi_{k , m}}{\|\pi_{k, m}\|} \in \Big[ \mathcal{X}_1  {\small{\times}} \ldots {\small{\times}} \mathcal{X}_{t-1}\Big]^{\varepsilon/2}$
gives for $k \geq T-t+2$,
$$
\|\pi_{k, m}\| \leq \frac{2}{\varepsilon}\Big(\mathfrak{Q}_t({\tilde x}_{1:t-1}^{k, m}, \xi_{m})- \mathfrak{Q}_t^{T-t+1}(x_{[n ]}^k, \xi_{m})    \Big),
$$
where we have used the fact that $\mathfrak{Q}_t^{k} \leq \mathfrak{Q}_t$ and 
$\mathfrak{Q}_t^{k+1}(\cdot, \xi_{m}) \geq \mathfrak{Q}_t^{k}(\cdot, \xi_{m})$, for all $k \in \mathbb{N}$.
In the proof of Proposition \ref{convexityrec}, we have shown that
for every $t=2,\ldots,T$, and $j=1,\ldots,M$, the function $\mathfrak{Q}_t(\cdot, \xi_{t, j})$ is finite on the compact
set $[\mathcal{X}_1 \small{\times} \ldots \small{\times}\mathcal{X}_{t-1}]^{\varepsilon/2}$.
Also, we have just shown that for every $t=2,\ldots,T,$ and $j=1,\ldots,M$,
the function $\mathfrak{Q}_t^{T-t+1}(\cdot, \xi_{t, j})$ is continuous on the compact
set $\mathcal{X}_1  {\small{\times}} \ldots {\small{\times}} \mathcal{X}_{t-1}$.
It follows that for every $k \geq T-t+1$ and node $m$, we have for $\pi_{k, m}$ the upper bound
\begin{equation} \label{uppboundpitkj}
\begin{array}{l}
\|\pi_{k, m}\| \leq \displaystyle \max_{t=2,\ldots,T,j=1,\ldots,M}\frac{2 M(t,j)}{\varepsilon}\mbox{ where }\\
M(t,j)=\displaystyle \max_{x_{1:t-1} \in [\mathcal{X}_1 \small{\times} \ldots \small{\times}\mathcal{X}_{t-1}]^{\varepsilon/2}}  \mathfrak{Q}_t(x_{1:t-1}, \xi_{t, j}) -  \min_{x_{1:t-1} \in  \mathcal{X}_1  {\small{\times}} \ldots {\small{\times}} \mathcal{X}_{t-1}} \mathfrak{Q}_t^{T-t+1}(x_{1:t-1}, \xi_{t, j}).
\end{array}
\end{equation}
\hfill
\end{proof}
\begin{rem} In the case when the cuts are computed in a backward pass using 
approximate recourse functions $\mathcal{Q}_{t+1}^k$ instead of $\mathcal{Q}_{t+1}^{k-1}$ for iteration $k$, we can guarantee that 
$\theta_t^k$ and $\beta_t^k$ are bounded for all $t=2,\ldots,T+1$, and $k \geq 1$ (for $k=0$, we have $\beta_t^k=0$ but $\theta_t^k=-\infty$ is not bounded for $t \leq T$).
\end{rem}

The following lemma will be useful in the sequel:
\begin{lemma}\label{eqqeqthetlemma} Consider the sequences $\mathcal{Q}_t^k, x_{[n_t^k]}^k$, and  $\theta_t^k$ generated by Algorithm 1.
Under Assumptions (H2), for $t=2,\ldots,T$, and for all $k \geq 1$, we have
\begin{equation}\label{eqqeqthet}
\mathcal{Q}_{t}^{k}( x_{[n_{t-1}^k]}^k)  = \theta_t^k.
\end{equation}
\end{lemma}
\begin{proof} We use the short notation $x_{1:t-1}^k = x_{[n_{t-1}^k]}^k$ and $n=n_{t-1}^k$.
Observe that by construction $\mathfrak{Q}_{t}^{k} \geq \mathfrak{Q}_{t}^{k-1}$ for every $t=2,\ldots,T+1$,
and every $k \in \mathbb{N}^{*}$. It follows that for fixed $0 \leq \ell \leq k$,
$$
\begin{array}{lll}
\theta_t^k  &= & \displaystyle \sup_{p \in \mathcal{P}_t} \;\displaystyle \sum_{m \in C(n)} p_{m} \Phi_{m} \mathfrak{Q}_{t}^{k-1}(x_{[n]}^k, \xi_{m})\\
&\geq & \displaystyle \sup_{p \in \mathcal{P}_t} \;\
\displaystyle \sum_{m \in C(n)} p_m \Phi_{m} \mathfrak{Q}_{t}^{\ell -1}(x_{1:t-1}^k, \xi_{m})\\
& = & \displaystyle \sup_{p \in \mathcal{P}_t} \;\
\displaystyle \sum_{m \in C(n_{t-1}^{\ell})} p_m \Phi_{m} \mathfrak{Q}_{t}^{\ell -1}(x_{1:t-1}^k, \xi_{m})\;\;\;(\mbox{since }(\xi_t) \mbox{ is interstage independent})\\
&\geq & \displaystyle \sup_{p \in \mathcal{P}_t} \;\displaystyle \sum_{m \in C(n_{t-1}^{\ell})} p_m \Phi_{m}   \left(\mathfrak{Q}_t^{\ell -1}(x_{1:t-1}^{\ell}, \xi_{m}) + \langle \pi_{\ell, m}, x_{1:t-1}^k - x_{1:t-1}^{\ell}\rangle \right)
\end{array}
$$
using the convexity of function $\mathfrak{Q}_t^{\ell-1}(\cdot, \xi_{m})$ and the fact that 
$\pi_{\ell, m}$ is a subgradient of this function at $x_{1:t-1}^{\ell}$.
Recalling that 
$$
\begin{array}{l}
\theta_t^\ell  = \rho_t \left( \mathfrak{Q}_{t}^{\ell -1}(x_{1:t-1}^{\ell}, \xi_t) \right)=  \displaystyle \sum_{m \in C( n_{t-1}^{\ell} )} p_{\ell, m} \Phi_{m} \mathfrak{Q}_t^{\ell-1}(x_{1:t-1}^\ell, \xi_{m})\;\;\mbox{and}\;\;\beta_t^\ell  = \displaystyle \sum_{m \in C( n_{t-1}^{\ell} )} p_{\ell, m} \Phi_{m}  \pi_{\ell, m},
\end{array}
$$
we get
$$
\begin{array}{lll}
\theta_t^k  &\geq &  \displaystyle \sum_{m \in C(n_{t-1}^{\ell})} p_{\ell, m} \Phi_{m}   \left(\mathfrak{Q}_t^{\ell -1}(x_{1:t-1}^{\ell}, \xi_{m}) + \langle \pi_{\ell, m}, x_{1:t-1}^k - x_{1:t-1}^{\ell}\rangle\right)\\
&\geq &  \theta_t^\ell + \langle \beta_t^\ell, x_{1:t-1}^{k}-x_{1:t-1}^\ell \rangle
\end{array}
$$
and 
$$
\mathcal{Q}_{t}^{k}(x_{1:t-1}^{k})  = \max\Big( \theta_t^k, \theta_t^\ell + \langle \beta_t^\ell, x_{1:t-1}^{k}-x_{1:t-1}^\ell \rangle,\;\ell=0,\ldots,k-1  \Big)=\theta_t^k.
$$
\hfill
\end{proof}
\begin{lemma}\label{lipforqtk} For $t=2,\ldots,T$, and $k \geq T-t+1$, the functions $\mathcal{Q}_t^k$ are
$L$-Lipschitz with $L$ given by
$$
\frac{2}{\varepsilon}\ \max_{t=2,\ldots,T} \left(\sup_{x_{1:t-1} \in [\mathcal{X}_1 \small{\times} \ldots \small{\times}\mathcal{X}_{t-1}]^{\varepsilon/2}} \mathcal{Q}_{t}(x_{1:t-1})-  \min_{x_{1:t-1} \in \mathcal{X}_1 \small{\times} \ldots \small{\times}\mathcal{X}_{t-1}} \mathcal{Q}_{t}^{T-t+1}(x_{1:t-1})\right).
$$
\end{lemma}
\begin{proof}This is an immediate consequence of \eqref{boundbeta} and \eqref{eqqeqthet}.\hfill
\end{proof}

\section{Convergence analysis for risk-averse multistage stochastic convex programs}\label{sec:relativelcr}

Theorem \ref{convalg1}  shows the convergence of the sequence $\mathfrak{Q}_{1}^k(x_0, \xi_1)$ 
to $\mathcal{Q}_{1}(x_0)$ and that 
any accumulation point of the sequence $(x_1^k)_{k \in \mathbb{N}^*}$ is an optimal solution 
of the first stage problem \eqref{firststagepb}.
\begin{thm}[Convergence analysis of Algorithm 1]\label{convalg1}
Consider the sequences of stochastic decisions $x_n^k$ and of recourse functions $\mathcal{Q}_ t^k$
generated by Algorithm 1 to solve dynamic programming equations  \eqref{definitionQt}-\eqref{defmathfrak}.
Let Assumptions (H1), (H2), and (H3) hold. Then
\begin{itemize}
\item[(i)] almost surely, for $t=2,\ldots,T+1$, the following holds:
$$
\mathcal{H}(t): \;\;\;\forall n \in {\tt{Nodes}}(t-1), \;\; \displaystyle \lim_{k \rightarrow +\infty} \mathcal{Q}_{t}(x_{[n]}^{k})-\mathcal{Q}_{t}^{k}(x_{[n]}^{k} )=0.
$$
\item[(ii)]
Almost surely, we have
$\displaystyle \lim_{k \rightarrow +\infty} \mathfrak{Q}_1^{k}(x_{0}, \xi_1)=\mathcal{Q}_{1}(x_0)$ and
any accumulation point of the sequence $(x_1^k)_{k \in \mathbb{N}^*}$ is an optimal solution of the first stage problem \eqref{firststagepb}.
\end{itemize}
\end{thm}
\begin{proof}
In this proof, all equalities and inequalities hold almost surely.
We show $\mathcal{H}(2),\ldots,\mathcal{H}(T+1)$, by induction backwards in time.
$\mathcal{H}(T+1)$ follows from the fact that $\mathcal{Q}_{T+1}=\mathcal{Q}_{T+1}^{k}=0$.
Now assume that $\mathcal{H}(t+1)$ holds for some $t \in \{2,\ldots,T\}$. We want to show that $\mathcal{H}(t)$ holds.
Take a node $n \in {\tt{Nodes}}(t-1)$. Let
$\mathcal{S}_n= \{k \geq 1 :  n_{t-1}^k =n\}$ be the set of iterations 
such that the sampled scenario passes through node $n$. 
Due to Assumption (H3), the set $\mathcal{S}_n$ is infinite.
We first show that
\begin{equation}\label{passesthroughn}
\displaystyle \lim_{k \rightarrow +\infty,\, k \in \mathcal{S}_n} \mathcal{Q}_{t}(x_{[n]}^{k})-\mathcal{Q}_{t}^{k}(x_{[n]}^{k} )=0.
\end{equation}
Take $k \in \mathcal{S}_n$. We have $n_{t-1}^k = n$ and using Lemma \ref{eqqeqthetlemma} and the definition of $\mathcal{Q}_t$, we get 
$$
0 \leq \mathcal{Q}_t(x_{[n]}^k)-\mathcal{Q}_t^k(x_{[n]}^k) = \mathcal{Q}_t(x_{[n]}^k)-\mathcal{Q}_t^k(x_{[n_{t-1}^k]}^k) =   \displaystyle \sup_{p \in \mathcal{P}_t} \;\displaystyle \sum_{m \in C(n)} p_m \Phi_{m} \mathfrak{Q}_t(x_{[n]}^k, \xi_{m})- \theta_t^k.  
$$
It follows that
\begin{equation}\label{eqfin1conv}
\begin{array}{l}
\mathcal{Q}_t(x_{[n]}^k)-\mathcal{Q}_t^k(x_{[n]}^k)\\
=\displaystyle \sup_{p \in \mathcal{P}_t} \;\displaystyle \sum_{m \in C(n)} p_m \Phi_{m} \mathfrak{Q}_t(x_{[n]}^k, \xi_{m})-
\displaystyle \sup_{p \in \mathcal{P}_t} \;\displaystyle \sum_{m \in C(n)} p_m \Phi_{m} \mathfrak{Q}_t^{k-1}(x_{[n]}^{k}, \xi_{m})  \mbox{ by definition of }\theta_t^k,\\
\leq \displaystyle \sup_{p \in \mathcal{P}_t} \;\displaystyle \sum_{m \in C(n)} p_m \Phi_{m}
\Big[ \mathfrak{Q}_t(x_{[n]}^{k}, \xi_{m}) - \mathfrak{Q}_t^{k-1}(x_{[n]}^{k}, \xi_{m})  \Big]\\
= \displaystyle \sup_{p \in \mathcal{P}_t} \;\displaystyle \sum_{m \in C(n)} p_m \Phi_{m}
\Big[ \mathfrak{Q}_t(x_{[n]}^{k}, \xi_{m})-  f_t(x_{[m]}^{k}, \Psi_m ) - \mathcal{Q}_{t+1}^{k-1}( x_{[m]}^k   ) \Big] \mbox{ by definition of }x_m^k,\\
= \displaystyle \sup_{p \in \mathcal{P}_t} \;\displaystyle \sum_{m \in C(n)} p_m \Phi_{m} \Big[ \mathfrak{Q}_t(x_{[n]}^{k}, \xi_{m}) - F_t(x_{[m]}^{k}, \Psi_m ) + \mathcal{Q}_{t+1}( x_{[m]}^k   ) - \mathcal{Q}_{t+1}^{k-1}( x_{[m]}^k   ) \Big]\\
\end{array}
\end{equation}
using the definition of $F_t$.
Observing that for every  $m \in C(n)$ and $k \in \mathcal{S}_n$ the decision $x_m^k \in X_t(x_0, x_{[n]}^k, \xi_m )$, we obtain, using definition \eqref{defmathfrak} of $\mathfrak{Q}_t$, that
$$
F_t(x_{[n]}^k, x_m^k, \Psi_m  ) = F_t(x_{[m]}^k, \Psi_m  ) \geq \mathfrak{Q}_{t}(x_{[n]}^k, \xi_{m}).
$$
Combining this relation with \eqref{eqfin1conv} gives for $k \in \mathcal{S}_n$
\begin{equation} \label{kinKn}
\begin{array}{l}
0 \leq \mathcal{Q}_t(x_{[n]}^k)-\mathcal{Q}_t^k(x_{[n]}^k) \leq  \displaystyle \sup_{p \in \mathcal{P}_t} \;\displaystyle \sum_{m \in C(n)} p_m \Phi_{m} \Big[\mathcal{Q}_{t+1}( x_{[m]}^k   ) - \mathcal{Q}_{t+1}^{k-1}( x_{[m]}^k   ) \Big].
\end{array}
\end{equation}                      
Using the induction hypothesis $\mathcal{H}(t+1)$, we have for every child node $m$ of node $n$ that
\begin{equation}\label{inductionm}
\displaystyle \lim_{k \rightarrow +\infty} \mathcal{Q}_{t+1}(x_{[m]}^{k})-\mathcal{Q}_{t+1}^{k}(x_{[m]}^{k} )=0.
\end{equation}
Now recall that $\mathcal{Q}_{t+1}$ is convex on the compact set $\mathcal{X}_1 \small{\times} \ldots \small{\times} \mathcal{X}_t$ (Proposition \ref{convexityrec}),
$x_{[m]}^k \in  \mathcal{X}_1 \small{\times} \ldots \small{\times} \mathcal{X}_t$ for every child node $m$ of node $n$,
and the functions $\mathcal{Q}_{t+1}^{k}, k\geq T-t+1$, are $L$-Lipschitz (Lemma \ref{lipforqtk})
with $\mathcal{Q}_{t+1} \geq \mathcal{Q}_{t+1}^{k} \geq \mathcal{Q}_{t+1}^{k-1}$ on $\mathcal{X}_1 \small{\times} \ldots \small{\times} \mathcal{X}_t$
(Lemma \ref{lipqtk}).
It follows that we can use Lemma A.1 in \cite{lecphilgirar12} to deduce from \eqref{inductionm} that
$$
\displaystyle \lim_{k \rightarrow +\infty   } \mathcal{Q}_{t+1}(x_{[m]}^{k})-\mathcal{Q}_{t+1}^{k-1}(x_{[m]}^{k} )=0.
$$
Combining this relation with \eqref{kinKn}, we obtain 
\begin{equation}\label{convfirst}
\lim_{k \rightarrow +\infty, k \in \mathcal{S}_n    } \mathcal{Q}_{t}(x_{[n]}^{k})-\mathcal{Q}_{t}^{k}(x_{[n]}^{k} )=0.
\end{equation}
To show $\mathcal{H}(t)$, it remains to show that 
\begin{equation}\label{inductionm3}
\displaystyle \lim_{k \rightarrow +\infty, k \notin \mathcal{S}_n    } \mathcal{Q}_{t}(x_{[n]}^{k})-\mathcal{Q}_{t}^{k}(x_{[n]}^{k} )=0.
\end{equation}
To show \eqref{inductionm3}, we proceed similarly to the end of the proof of Theorem 3.1 in \cite{lecphilgirar12}, by contradiction and using the Strong Law of Large Numbers.
For the sake of completeness, we apply here these arguments in our context, where the notation and 
the convergence statement $\mathcal{H}(t)$ is different from \cite{lecphilgirar12}.
If \eqref{inductionm3} does not hold, there exists $\varepsilon>0$ such that there is an infinite number of iterations $k \in  \mathbb{N}$ satisfying
$\mathcal{Q}_{t}(x_{[n]}^{k})-\mathcal{Q}_{t}^{k}(x_{[n]}^{k} ) \geq \varepsilon$. Since $\mathcal{Q}_{t}^{k} \geq \mathcal{Q}_{t}^{k-1}$, there is also an infinite number
of iterations belonging to the set
$$
\mathcal{K}_{n, \varepsilon}= \{k \in \mathbb{N} :  \mathcal{Q}_{t}(x_{[n]}^{k})-\mathcal{Q}_{t}^{k-1}(x_{[n]}^{k} ) \geq \varepsilon \}.
$$
Consider the stochastic processes $(w_n^k)_{k \in \mathbb{N}^*}$ and
$(y_n^k)_{k \in \mathbb{N}^*}$ 
where $w_n^k = 1_{k \in  \mathcal{K}_{n, \varepsilon}}$ and $y_n^k = 1_{k \in  \mathcal{S}_{n}}$, i.e., 
$y_n^k$ takes the value $1$ if
node $n$ belongs to the sampled scenario for iteration $k$ (when $n_{t-1}^k = n$) and $0$ otherwise.
Assumption (H3) implies that random variables $(y_n^k)_{k \in \mathbb{N}^*}$ are independent
and setting ${\tilde{\mathcal{F}}}_k=\sigma(w_n^1, \ldots,w_n^{k}, y_n^1, \ldots, y_n^{k-1})$,
by definition of $x_{[n]}^j$ and $\mathcal{Q}_{t}^j$ that
$y_n^k$ is independent on $( (x_{[n]}^j,j=1,\ldots,k), (\mathcal{Q}_{t}^j,j=1,\ldots,k-1))$ and thus 
of ${\tilde{\mathcal{F}}}_k$. If $z^j$ is the $j$th element in the set
$\{y_n^k : k \in \mathcal{K}_{n, \varepsilon}\}$, using Lemma A.3 in \cite{lecphilgirar12}, we obtain that 
random variables $z^j$ are i.i.d. and have the distribution of $y_n^1$.
Using the Strong Law of Large Numbers, we get
$$
\frac{1}{N} \sum_{j=1}^N z^j \xrightarrow{N \rightarrow +\infty} \mathbb{E}[z^1] = \mathbb{E}[y_n^1]=\mathbb{P}(y_n^1 >0)
\stackrel{(H3)}{>}0.
$$
Relation \eqref{convfirst} and Lemma A.1 in \cite{lecphilgirar12} imply that
$\lim_{k \rightarrow +\infty, k \in \mathcal{S}_n    } \mathcal{Q}_{t}(x_{[n]}^{k})-\mathcal{Q}_{t}^{k-1}(x_{[n]}^{k} )=0$.
It follows that the set
$\mathcal{K}_{n, \varepsilon} \cap \mathcal{S}_n = \mathcal{K}_{n, \varepsilon} \cap \{k \in \mathbb{N}^* : y_n^k =1\}$ is finite.
This implies 
$$
\frac{1}{N} \sum_{j=1}^N z^j \xrightarrow{N \rightarrow +\infty} 0,
$$
which yields the desired contradiction and achieves the proof of (i).
\par (ii) 
By definition of $\mathfrak{Q}_1^{k-1}$ (see Algorithm 1), we have
\begin{equation}\label{solfirststapp}
\begin{array}{lll}
\mathfrak{Q}_1^{k-1}(x_{0}, \xi_1)&= &   
f_1(x_{[n_1]}^k, \Psi_1) + \mathcal{Q}_2^{k-1}(x_{[n_1]}^k)
=F_1(x_{[n_1]}^k, \Psi_1)-\mathcal{Q}_2(x_{[n_1]}^k)+\mathcal{Q}_2^{k-1}(x_{[n_1]}^k).
\end{array}
\end{equation}
Since $x_{[n_1]}^k \in X_1(x_0,\xi_1)$ we have $F_1(x_{[n_1]}^k, \Psi_1) \geq \mathfrak{Q}_1(x_0, \xi_1)$.
Together with \eqref{solfirststapp}, this implies 
\begin{equation}\label{finaliii}
0 \leq \mathfrak{Q}_1(x_{0}, \xi_1)-\mathfrak{Q}_1^{k-1}(x_{0}, \xi_1)\leq \mathcal{Q}_2(x_{[n_1]}^k)-\mathcal{Q}_2^{k-1}(x_{[n_1]}^k).
\end{equation}
Using $\mathcal{H}(2)$ from item (i) and Lemma A.1 in \cite{lecphilgirar12}, this implies
\begin{equation}\label{convq2k}
\displaystyle \lim_{k \rightarrow +\infty  } \mathcal{Q}_{2}(x_{[n_1]}^{k})-\mathcal{Q}_{2}^{k-1}(x_{[n_1]}^{k} )=0.
\end{equation}
Plugging this relation into \eqref{finaliii}, we get
$\displaystyle \lim_{k \rightarrow +\infty } \mathfrak{Q}_1^{k-1}(x_{0}, \xi_1) = \mathfrak{Q}_1(x_{0}, \xi_1)$.

Recalling that $n_1$ is the node associated to the first stage,
consider now an accumulation point $x_{n_1}^*$ of the sequence $(x_{n_1}^k)_{k \in \mathbb{N}}$.
There exists a set $K$ such that the sequence $(x_{n_1}^k)_{k \in K}$ converges to $x_{n_1}^*$.
By definition
of $x_{n_1}^k$ and since $\Psi_{n_1} = \Psi_1$, we get
\begin{equation}\label{convfirststage1}
\begin{array}{lll}
 f_1(x_{n_1}^k, \Psi_1) + \mathcal{Q}_2^{k-1}(x_{n_1}^k)=\mathfrak{Q}_1^{k-1}(x_0, \xi_1).
\end{array} 
\end{equation}
Using \eqref{convq2k} and the continuity of $\mathcal{Q}_2$ on $\mathcal{X}_1$, we have
$$
\lim_{k \rightarrow +\infty,\, k \in K}\,\mathcal{Q}_2^{k-1}(x_{n_1}^k)=\mathcal{Q}_2( x_{n_1}^*).
$$
Taking the limit in \eqref{convfirststage1} when $k \rightarrow +\infty$ with $k \in K$ and using the lower semicontinuity of $f_1$, we obtain
$$
f_1(x_{n_1}^{*}, \Psi_1) + \mathcal{Q}_2( x_{n_1}^* )=F_{1}(x_{n_1}^*, \Psi_1) \leq \mathcal{Q}_1(x_0).
$$
Since for every $k \in K$,  $x_{n_1}^k$ is feasible for the first stage problem, so is $x_{n_1}^*$ (due to the lower semicontinuity of $g_1(x_0,\cdot,\Psi_1)$
and the compactness of $\mathcal{X}_1$, the set $X_1(x_0, \xi_1)$ is closed) and 
$x_{n_1}^*$ is an optimal solution to the first stage problem.\hfill
\end{proof}

\begin{rem}\label{importantremark}
In Algorithm 1, decisions are computed at every iteration for all the nodes of the scenario tree.
However, in practice, decisions will only be computed for the nodes of the sampled scenarios and their children nodes
(such is the case of SDDP).
This variant of Algorithm 1 will build the same cuts and compute the same decisions for the nodes of the
sampled scenarios as Algorithm 1. For this variant, for a node $n$, the decision variables $(x_n^k)_k$ are defined for
an infinite subset ${\tilde{\mathcal{S}}}_{n}$ of iterations where the sampled scenario passes through the parent node of node $n$, i.e., 
${\tilde{\mathcal{S}}}_{n}=\mathcal{S}_{\mathcal{P}(n)}$.
With this notation, for this variant of Algorithm 1, applying Theorem \ref{convalg1}-(i), we get for $t=2,\ldots,T+1$,
\begin{equation}\label{variantalg1}
\mbox{for all }n \in {\tt{Nodes}}(t-1), \lim_{k \rightarrow +\infty, k \in {\tilde{\mathcal{S}}}_{n}} \mathcal{Q}_{t}(x_{[n]}^{k})-\mathcal{Q}_{t}^{k}(x_{[n]}^{k} )=0
\end{equation}
almost surely, while Theorem \ref{convalg1}-(ii) still holds.
\end{rem}

\begin{rem}\label{remassumptions}
If for a given stage $t$, $\mathcal{X}_t$ is a polytope and we do not have the nonlinear constraints
given by constraint functions $g_t$ (i.e., the constraints for this stage are linear), then the conclusions of 
Proposition \ref{convexityrec} and Lemmas \ref{lipqtk}, \ref{eqqeqthetlemma}, and \ref{lipforqtk} hold
and thus Theorem \ref{convalg1} holds too under weaker assumptions. 
More precisely, for such stages $t$, we assume (H2)-1),
(H2)-2), and instead of (H2)-4), (H2)-5), the weaker assumption (H2)-3'):\\
\par (H2)-3') There exists $\varepsilon>0$ such that:
\begin{itemize}
\item[3.1')] for every $j=1,\ldots,M$,
$
\Big[\mathcal{X}_1 {\small{\times}} \ldots {\small{\times}} \mathcal{X}_{t-1} \Big]^{\varepsilon}{\small{\times}} \mathcal{X}_{t}  \subset \mbox{dom} \; f_t\Big(\cdot, \Psi_{t, j}\Big);
$
\item[3.2')] for every $j=1,\ldots,M$, for every
$x_{1:t-1} \in \mathcal{X}_1 {\small{\times}} \ldots {\small{\times}} \mathcal{X}_{t-1}$,
the set $X_t(x_{0:t-1}, \xi_{t, j})$ is nonempty.
\end{itemize}
\end{rem}

\section{Convergence analysis for risk-averse multistage stochastic linear programs without relatively complete recourse} \label{sec:withoutrelativelcr}

In this section, we consider the case when $g_t$ is affine and $f_t$ is linear.
We replace assumption (H2)-1) by $\mathcal{X}_t=\mathbb{R}_{+}^{n}$ and we do not make
Assumptions (H2)-4)-5). More precisely, instead of \eqref{defmathfrak}, we consider the following 
dynamic programming equations corresponding to multistage stochastic linear programs
that do not satisfy the relatively complete recourse assumption:
we set $\mathcal{Q}_{T+1} \equiv 0$ and for
$t=2,\ldots,T$, we define $\mathcal{Q}_{t}(x_{1:t-1})=\rho_t\Big(\mathfrak{Q}_{t}(x_{1:t-1}, \xi_t)\Big)$
now with
\begin{equation} \label{defmathfrak1}
\begin{array}{lll}
\mathfrak{Q}_{t}(x_{1:t-1}, \xi_t)&=&
\left\{ 
\begin{array}{l}
\displaystyle \inf_{x_t}\;F_t(x_{1:t}, \Psi_t ):= \Psi_{t}^{\transp} x_{1:t}  + \mathcal{Q}_{t+1}(x_{1:t})\\
\displaystyle \sum_{\tau=0}^{t} \; A_{t, \tau} x_{\tau} = b_t,\;x_t \geq 0,
\end{array}
\right.
\\
&=&\left\{
\begin{array}{l}
\displaystyle \inf_{x_t} \;F_t(x_{1:t}, \Psi_t)\\
x_t \in X_t (x_{0:t-1}, \xi_t ).
\end{array}
\right.
\end{array}
\end{equation}
At the first stage, we solve
\begin{equation} \label{defmathfrak2}
\begin{array}{lll}
\left\{ 
\begin{array}{l}
\displaystyle \inf_{x_1}\;F_1(x_{1}, \Psi_1) := \Psi_{1}^{\transp} x_{1}  + \mathcal{Q}_{2}(x_{1})\\
x_1 \in X_1 (x_{0}, \xi_1 )
=\{x_1 \geq 0, A_{1,1} x_1 = b_1 - A_{1,0} x_0\},
\end{array}
\right.
\end{array}
\end{equation}
with optimal value denoted by $\mathcal{Q}_1(x_0)=\mathfrak{Q}_1(x_0, \xi_1)$.
If we apply Algorithm 1 to solve \eqref{defmathfrak2} (in the sense of Theorem \ref{convalg1}), since Assumption (H2)-4) does not hold, it is possible that 
one of the problems \eqref{defxtkj} to be solved in the forward passes is infeasible. 
In this case, $x_{[n]}^k$ is not a feasible sequence of states from stage $1$ to stage $t-1$
and we build a separating hyperplane separating $x_{n}^k$ and the set of states that are feasible
at stage $t-1$ (those for which there exist sequences of decisions on any future scenario, assuming 
that problem \eqref{defmathfrak2}  is feasible).
The construction of
feasibility cuts for the nested decomposition algorithm is described in \cite{birge-louv-book}.
Feasibility cuts for sampling based decomposition algorithms were introduced 
in \cite{guiguescoap2013}. This latter reference also discusses how to share feasibility cuts
among nodes of the same stage for some interstage independent processes and stochastic programs.
In the case of problem \eqref{defmathfrak1}, before solving problems \eqref{defxtkj} for all $m \in C(n)$ in the forward pass,
setting $n=n_{t-1}^k$, we solve for every $m \in C(n)$ the optimization problem
\begin{equation} \label{phaseI}
\begin{array}{l}
\displaystyle \min_{x_t, y_1, y_2} \;e^\transp ( y_1 + y_2 )\\
\displaystyle A_{t, m} x_{t} + y_1 - y_2  = b_{m}- \sum_{\tau=0}^{t-1} \; A_{\tau, m} x_{\mathcal{P}^{t-\tau} (m)}^k, \hspace*{0.85cm}[\pi]\\
(x_{[n]}^k)^\transp {\tilde \beta}_{t+1, 1}^{\ell} + x_t^\transp {\tilde \beta}_{t+1, 2}^{\ell} \leq {\tilde \theta}_{t+1}^{\ell},\;\ell=1,\ldots,K_t,\;\hspace*{0.38cm}[\tilde \pi]\\
x_t, y_1, y_2 \geq 0,
\end{array}
\end{equation}
where $e$ is a vector of ones.
In the above problem, we have denoted by respectively $\pi$ and $\tilde \pi$ optimal Lagrange multipliers
for the first and second set of constraints\footnote{We suppressed the dependency with respect to $n, k$ to alleviate notation.}.
If ${\tilde{\mathfrak{Q}}}_t(x_{[n]}^k, \xi_{m})$
is the optimal value of \eqref{phaseI}, noting that 
${\tilde{\mathfrak{Q}}}_t(\cdot, \xi_{m})$ is convex with
$s=[A_{1, m}^\transp; \ldots;A_{t-1, m}^\transp] \pi + 
[{\tilde \beta}_{t+1, 1}^{1}, \ldots,{\tilde \beta}_{t+1, 1}^{K_t}]{\tilde \pi}$ belonging to the subdifferential of  
${\tilde{\mathfrak{Q}}}_t(\cdot, \xi_{m})$ at $x_{[n]}^k$, if 
$x_{1:t-1}$ is feasible then
\begin{equation} \label{feascuts}
{\tilde{\mathfrak{Q}}}_t(x_{1:t-1}, \xi_{m})=0 \geq  {\tilde{\mathfrak{Q}}}_t(x_{[n]}^{k}, \xi_{m}) + s^\transp (x_{1:t-1}-x_{[n]}^k).
\end{equation}
Inequality \eqref{feascuts} defines a feasibility cut for $x_{1:t-1}$ of the form
\begin{equation} \label{fcuts}
x_{1:t-1}^\transp {\tilde \beta}_{t}^{\ell} \leq {\tilde \theta}_{t}^{\ell}
\end{equation}
with 
\begin{equation} \label{fcutsform}
{\tilde \beta}_{t}^{\ell} = [({\tilde \beta}_{t, 1}^{\ell})^\transp ; ({\tilde \beta}_{t, 2}^{\ell})^\transp   ] = s  \mbox{ and }{\tilde 
\theta}_{t}^{\ell}=-{\tilde{\mathfrak{Q}}}_t(x_{[n]}^{k}, \xi_{m}) + s^\transp x_{[n]}^k,
\end{equation}
where ${\tilde \beta}_{t, 1}^{\ell}$ (resp. ${\tilde \beta}_{t, 2}^{\ell}$) is the vector containing the first $n(t-2)$ (resp. last $n$) components of ${\tilde \beta}_{t}^{\ell}$.
Incorporating these cuts in the forward pass
of Algorithm 1, we obtain Algorithm 2.\\
\rule{\linewidth}{1pt}
\par {\textbf{Algorithm 2: Multistage stochastic decomposition algorithm to solve \eqref{defmathfrak1} without the relatively complete recourse assumption.}}\\
\par {\textbf{Initialization.}} Set $k=1$ (iteration count), Out=0 (Out will be 1 if the problem is infeasible),
$K_t=0$ (number of feasibility cuts at stage $t$), $\mathcal{Q}_{t}^0 \equiv-\infty$ for $t=2,\ldots,T$, and
$\mathcal{Q}_{T+1}^0 \equiv 0$, i.e., set $\theta_{T+1}^0=0$, $\beta_{t+1}^0=0$ for $t=1,\ldots,T$, 
and $\theta_{t+1}^0$  to $-\infty$ for $t=1,\ldots,T-1$.\\
\par {\textbf{While}} Out=0\\
\hspace*{1.15cm}Sample a set of $T+1$ nodes $(n_0^k, n_1^k, \ldots, n_T^k)$ such that
$n_0^k$ is the root node, $n_1^k = n_1$ is the node\\
\hspace*{1.15cm}corresponding to the first stage, and for every
$t=1,\ldots,T$, node $n_t^k$ is a child node of node $n_{t-1}^k$.\\
\hspace*{1.15cm}This set of nodes is associated to a sample $({\tilde \xi}_1^k, {\tilde \xi}_2^k,\ldots, {\tilde \xi}_T^k)$
of $(\xi_1, \xi_2,\ldots, \xi_T)$, realization of random
\hspace*{1.15cm}variables $(\xi_1^k, \ldots,\xi_T^k)$. \\
\hspace*{1.15cm}Set $t=1$, $n=n_{t-1}^k = n_0$, and $x_n^k=x_0$.\\
\hspace*{1.15cm}{\textbf{While}} ($t<=T$) and (Out=0)\\
\hspace*{1.6cm}Set OutAux$=0$.\\
\hspace*{1.6cm}{\textbf{While}} there remains a nonvisited child of $n$ and (OutAux$=0$),\\
\hspace*{2cm}Take for $m$ a nonvisited child of $n$.\\
\hspace*{2cm}Solve problem \eqref{phaseI}.\\
\hspace*{2cm}{\textbf{If}} the optimal value of \eqref{phaseI} is positive then\\
\hspace*{2.5cm}{\textbf{If}} $t=1$ then\\
\hspace*{2.8cm}Out=1, OutAux$=1$ //the problem is infeasible\\
\hspace*{2.5cm}{\textbf{Else}}\\
\hspace*{2.8cm}Compute \eqref{fcutsform} to build feasibility cut \eqref{fcuts}.\\
\hspace*{2.8cm}Increase $K_{t-1}$ by one.\\
\hspace*{2.8cm}Set $n=n_{t-2}^k$ and OutAux$=1$.\\
\hspace*{2.8cm}Decrease $t$ by one.\\
\hspace*{2.5cm}{\textbf{End if}}\\
\hspace*{2cm}{\textbf{End If}}\\
\hspace*{1.6cm}{\textbf{End While}}\\
\hspace*{1.6cm}{\textbf{If}} OutAux$=0$,\\
\hspace*{2cm}Setting $m=n_t^k$, compute  an optimal solution $x_m^k$ of
\begin{equation}\label{forward1}
\left\{
\begin{array}{l}
\displaystyle \inf_{x_t} \;([x_{[n]}^k; x_t])^\transp \Psi_m + z\\
\displaystyle  \sum_{\tau=0}^{t-1} \; A_{\tau, m} x_{\mathcal{P}^{t-\tau} (m)}^k + A_{t, m} x_t = b_m,\;\;\;\;\;\;\hspace*{0.35cm}[\pi_{k, m, 1}]\\
z \geq \theta_{t+1}^{\ell} + \langle \beta_{t+1, 1}^{\ell}, x_{[n]}^{k}-x_{[n_{t-1}^{\ell}]}^{\ell} \rangle \\
\;\;\;\;\;\;+ \langle \beta_{t+1, 2}^{\ell}, x_{t}-x_{n_t^{\ell}}^{\ell} \rangle,\;\;0 \leq \ell \leq k-1,\;\;[\pi_{k, m, 2}]\\
([x_{[n]}^k; x_t])^\transp {\tilde \beta}_{t+1}^{\ell} \leq {\tilde \theta}_{t+1}^{\ell},\;\ell=1,\ldots,K_t,\;\hspace*{0.23cm}[\pi_{k, m, 3}]\\
x_t \geq 0.
\end{array}
\right.
\end{equation}
\hspace*{2cm}In the above problem, we have
denoted by $\pi_{k, m, 1}, \pi_{k, m, 2}$,  and $\pi_{k, m, 3}$, 
the optimal Lagrange
\hspace*{2cm}multipliers associated with respectively the first, second, and third group of constraints.\\
\hspace*{2cm}Increase $t$ by one and set $n=n_{t-1}^k$.\\
\hspace*{1.6cm}{\textbf{End If}}\\
\hspace*{1.15cm}{\textbf{End While}}\\
\hspace*{1.15cm}{\textbf{If }}Out=0\\
\hspace*{1.6cm}{\textbf{For }} $t=2,\ldots,T$,\\
\hspace*{1.8cm}{\textbf{For }}each child node $m$ of $n=n_{t-1}^k$ with $m \neq n_t^k$, \\
\hspace*{2.55cm}compute an optimal solution $x_m^k$ and the optimal value $\mathfrak{Q}_t^{k-1}(x_{[n]}^{k}, \xi_{m})$ \\
\hspace*{2.55cm}of optimization problem \eqref{forward1}.\\
\hspace*{1.8cm}{\textbf{End For}}\\
\hspace*{1.8cm}For every child node $m$ of $n$, compute 
$$
\begin{array}{lll}
\pi_{k, m}& = & \Psi_{m}(1:(n-1)t) + \left(\begin{array}{c}A_{1, m}^\transp\\\vdots\\A_{t-1, m}^\transp \end{array}   \right) \pi_{k, m, 1}\\
&&+[\beta_{t+1, 1}^{0}, \ldots \beta_{t+1, 1}^{k-1}]\pi_{k, m, 2} +[{\tilde \beta}_{t+1, 1}^{1}, \ldots {\tilde \beta}_{t+1, 1}^{K_t}]\pi_{k, m, 3}.
\end{array}
$$
\hspace*{1.8cm}Compute $(p_{k, m})_{m \in C(n)}$ such that 
$$
\begin{array}{lll}
\rho_t \left( \mathfrak{Q}_{t}^{k-1}(x_{[n]}^{k}, \xi_t )  \right)&=& \displaystyle \sup_{p \in \mathcal{P}_t} \;
\displaystyle \sum_{m \in C(n)} p_m \Phi_{m} \mathfrak{Q}_{t}^{k-1}(x_{[n]}^k, \xi_{m})\\
&=&\displaystyle \sum_{m \in C(n)} p_{k, m} \Phi_{m} \mathfrak{Q}_{t}^{k-1}(x_{[n]}^{k}, \xi_{m})
\end{array}
$$
\hspace*{2cm}and coefficients
\begin{equation}\label{formulathetak2}
\begin{array}{l}
\theta_t^k  = \rho_t \left( \mathfrak{Q}_{t}^{k-1}(x_{[n]}^{k}, \xi_t) \right)=  \displaystyle \sum_{m \in C(n)} p_{k, m} \Phi_{m} \mathfrak{Q}_t^{k-1}(x_{[n]}^k, \xi_{m}),\;\;\;\beta_t^k  = \displaystyle \sum_{m \in C(n)} p_{k, m} \Phi_{m}  \pi_{k, m},
\end{array}
\end{equation}
\hspace*{2cm}making up the new approximate recourse function
$$
\mathcal{Q}_{t}^{k}(x_{1:t-1})  = \displaystyle \max_{0 \leq \ell \leq k}\;\Big( \theta_t^{\ell} + \langle \beta_t^{\ell}, x_{1:t-1}-x_{[n_{t-1}^{\ell}]}^{\ell} \rangle \Big).
$$
\hspace*{1.6cm}{\textbf{End For}}\\
\hspace*{1.6cm}Compute $\theta_{T+1}^{k}=0$ and $\beta_{T+1}^{k}=0$.\\
\hspace*{1.15cm}{\textbf{End If}}\\
\hspace*{1.15cm}Increase $k$ by one.\\
{\textbf{End While}}.\\
\rule{\linewidth}{1pt}
Theorem \ref{convalg2} which is a convergence analysis of
Algorithm 2 is a corollary of the convergence analysis of Algorithm 1 from Theorem \ref{convalg1}:
\begin{thm}[Convergence analysis of Algorithm 2] \label{convalg2} 
Let Assumptions (H1) and (H3) hold and assume that
\begin{itemize}
\item[(H2')] for every $t=1,\ldots,T$, for every realization $({\tilde \xi}_1, {\tilde \xi}_2,\ldots,{\tilde \xi}_t)$
of $(\xi_1, \xi_2,\ldots,\xi_t)$, for every sequence of feasible decisions $x_{0:t-1}$ on that scenario,
i.e., satisfying $x_\tau \in X_{\tau}(x_{0:\tau-1}, {\tilde \xi}_{\tau})$ for $\tau=1,\ldots,t-1$, 
the set $X_t(x_{0:t-1}, {\tilde \xi}_t)$
is bounded and nonempty. 
\end{itemize}
Then either Algorithm 2 terminates reporting that the problem is infeasible or
for $t=2,\ldots,T$, \eqref{variantalg1} holds almost surely and Theorem \ref{convalg1}-(ii) holds.
\end{thm}
\begin{proof} Due to Assumption (H2'), recourse
functions $\mathcal{Q}_t$ are convex polyhedral and Lipschitz continuous.
Moreover, Assumption (H2') also guarantees that
\begin{itemize}
\item[(a)] all linear programs \eqref{forward1}  are feasible and have bounded primal and dual
feasible sets. As a result, functions $(\mathcal{Q}_t^k)_{t, k}$ are also Lipschitz continuous convex
and polyhedral. 
\item[(b)] The feasible set of \eqref{phaseI} is bounded and nonempty. 
\end{itemize}
From (b) and Assumption (H1), we obtain that there is only a finite number 
of different feasibility cuts. From the definition of these feasibility cuts,
the feasible set of \eqref{phaseI} contains the first stage feasible set.
As a result, if \eqref{phaseI} is not feasible, there is no solution to 
\eqref{defmathfrak1}. Otherwise, since only a finite
number of different feasibility cuts can be generated, after some iteration
$k_0$ no more feasibility cuts are generated. In this case, 
after iteration $k_0$, Algorithm 2 is the variant of Algorithm 1 
described in Remark \ref{importantremark} and the proof can be achieved combining the
proof of Algorithm 1 and Remark \ref{importantremark}.\hfill
\end{proof}

\section{Convergence analysis with interstage dependent processes} \label{idp}

Consider a problem of form \eqref{pbinit0} with an interstage dependent process $(\xi_t)$,  
and let Assumption (H2) hold.
We assume that the stochastic process $(\xi_t)$ is discrete 
with a finite number of realizations at each stage.
The realizations of the process over the optimization period can still be represented by a finite scenario
tree with the root node $n_0$ associated to a fictitious stage 0 with decision $x_0$ taken at that
node. The unique child node $n_1$ of this root node corresponds to
the first stage (with $\xi_1$ deterministic).
In addition to the notation introduced in Section \ref{sec:relativelcr}, we also define
$\tau_n$ to be the stage associated to node $n$.

For interstage dependent processes, Algorithm 1 can be extended in two ways.
For some classes of processes, we can add in the state
vectors past process values while preserving the convexity of the recourse
functions. We refer to \cite{morton}, \cite{guiguescoap2013} for more details.
The convergence of Algorithm 1 applied to the corresponding dynamic programming
equations can be proved following the developments of Sections \ref{framework}
and \ref{sec:relativelcr}.

It is also possible to deal with more general interstage dependent processes
as in \cite{lecphilgirar12}. However, in this case, recourse functions are not linked to stages
but to the nodes of the scenario tree.
In this context, we associate to each node $n$ of the tree a coherent risk measure
$\rho_n: \mathbb{R}^{|C(n)|}\rightarrow \mathbb{R}$ and risk measure $\rho_{t+1|\mathcal{F}_t}$ in formulation
\eqref{pbinit0} is given by the collection of the risk measures $(\rho_n)_{n\;:\;\tau_n =t}$.
More precisely, we consider the following dynamic programming equations: for every node $n$
which is neither the root node nor a leaf, using the dual representation of risk measure $\rho_n$,
we define the recourse function
\begin{equation} \label{dpinterstagedep}
\mathcal{Q}_n( x_{[n]} )=\rho_n \Big( \mathfrak{Q}_n( x_{[n]}, (\xi_{m})_{m \in C(n)} )  \Big)= 
\displaystyle \sup_{p \in \mathcal{P}_n} \; \sum_{m \in C(n)} \; p_m \Phi_m  \mathfrak{Q}_n( x_{[n]}, \xi_{m})
\end{equation}
for some convex subset $\mathcal{P}_n$ of
$$
\mathcal{D}_n=\{p \in \mathbb{R}^{|C(n)|} : p \geq 0,\;\;\sum_{m \in C(n)}\;p_m \Phi_{m} =1  \},
$$
where $\mathfrak{Q}_n( x_{[n]}, \xi_{m})$ is given by
\begin{equation}\label{eqqnsnxim}
\left\{
\begin{array}{l}
\displaystyle \inf_{x_m} \;F_{\tau_m}(x_{[n]}, x_m, \Psi_{ m})\\
x_m \in \mathcal{X}_{\tau_m},\; g_{\tau_m}(x_0, x_{[n]}, x_{m}, \Psi_{m}) \leq 0,\\
\mbox{[ } A_{0, m}, \ldots, A_{\tau_m, m}\mbox{]} \mbox{[}x_0; x_{[n]}; x_{m}\mbox{]}=b_m
\end{array}
\right.
=
\left\{
\begin{array}{l}
\displaystyle \inf_{x_m} \;F_{\tau_m}(x_{[n]}, x_m, \Psi_{m})\\
x_m \in X_{\tau_m}\Big( x_0, x_{[n]}, \xi_{m}  \Big)
\end{array}
\right.
\end{equation}
with
$$
F_{\tau_m}(x_{[n]}, x_m, \Psi_{m})=f_{\tau_m}(x_{[n]}, x_m, \Psi_{m}) + \mathcal{Q}_{m}(x_{[n]}, x_{m}).
$$
If $n$ is a leaf node then $\mathcal{Q}_n \equiv 0$.
For the first stage, we solve problem \eqref{eqqnsnxim} with $n=n_0$ and $m=n_1$,
with optimal value denoted by $\mathcal{Q}_{n_0}(x_{0})=\mathfrak{Q}_{n_0}(x_{0}, \xi_{n_1})$
where $\xi_{n_1}=\xi_{1}$.

Algorithm 3 solves these dynamic programming equations building at iteration $k$ polyhedral lower approximation
$\mathcal{Q}_n^k$ of $\mathcal{Q}_n$ where
$$
\mathcal{Q}_n^k ( x_{[n]} ) = \max\Big(\theta_n^{\ell} + \langle  \beta_n^\ell, x_{[n]} - x_{[n]}^{\ell}\rangle , 0 \leq  \ell \leq k \Big)
$$ for all node $n \in \mathcal{N}\backslash \{n_0\}$.

If node $n$ is not a leaf, we introduce the function $\mathfrak{Q}_n^{k-1}$ such that
$\mathfrak{Q}_n^{k-1}( x_{[n]}, \xi_{m})$ is given by
\begin{equation} \label{equationalgorithm3}
\left\{
\begin{array}{l}
\displaystyle \inf_{x_m} \;F_{\tau_m}^{k-1}(x_{[n]}, x_m, \Psi_{m}):=f_{\tau_m}(x_{[n]}, x_m, \Psi_{m}) + \mathcal{Q}_{m}^{k-1}(x_{[n]}, x_{m})\\
g_{\tau_m}(x_0, x_{[n]}, x_{m}, \Psi_{m}) \leq 0,\\
\mbox{[ } A_{0, m}, \ldots, A_{\tau_m, m}\mbox{]} \mbox{[}x_0; x_{[n]}; x_{m}\mbox{]}=b_m\\
x_m \in \mathcal{X}_{\tau_m}.
\end{array}
\right.
\end{equation}
Next, we write $\mathfrak{Q}_n^{k-1}( x_{[n]}, \xi_{m})$ under the form
\begin{equation} \label{equationalgorithm3bis}
\left\{
\begin{array}{lll}
\displaystyle \inf_{x_m} \;f_{\tau_m}(x_{[n]}, x_m, \Psi_{m})+z&\\
g_{\tau_m}(x_0, x_{[n]}, x_{m}, \Psi_{m}) \leq 0,& [\pi_{k, m, 1}]\\
\mbox{[ } A_{0, m}, \ldots, A_{\tau_m, m}\mbox{]} \mbox{[}x_0; x_{[n]}; x_{m}\mbox{]}=b_m &  [\pi_{k, m, 2}]\\
z \geq \theta_{m}^{\ell} + \langle \beta_{m, 1}^{\ell}, x_{[n]}-x_{[n]}^{\ell} \rangle + \langle \beta_{m, 2}^{\ell}, x_{m}-x_{m}^{\ell} \rangle,\;\;\ell \leq k-1,   & [\pi_{k, m, 3}]\\
x_m \in \mathcal{X}_{\tau_m},&
\end{array}
\right.
\end{equation}
where $\beta_{m, 1}^{\ell}$ (resp. $\beta_{m, 2}^{\ell}$) contains the first $n \tau_n $ (resp. last $n$)
components of $\beta_{m}^{\ell}$.
In the above problem, we have
denoted by $\pi_{k, m, 1}, \pi_{k, m, 2}$,  and $\pi_{k, m, 3}$ the optimal Lagrange multipliers associated with
respectively the first, second, and third group of constraints. Finally, if $n$ is a leaf, $\mathfrak{Q}_n^{k-1}=0$.

For $n=n_0$, the optimal value of \eqref{equationalgorithm3} is denoted by
$\mathfrak{Q}_{n_0}^{k-1}( x_{0}, \xi_{1})$.\\
\rule{\linewidth}{1pt}
\par {\textbf{Algorithm 3: Multistage stochastic decomposition algorithm to solve \eqref{dpinterstagedep} for interstage dependent processes.}}\\
\par {\textbf{Initialization.}} Set  $\mathcal{Q}_{n}^0 \equiv 0$ for all leaf node $n$ and
$\mathcal{Q}_{n}^0 \equiv-\infty$ for all other node $n$.\\

\par {\textbf{For $k=1,2,\ldots,$}}
\par \hspace*{1cm}Sample a scenario $({\tilde \xi}_1^k, {\tilde \xi}_2^k, \ldots, {\tilde \xi}_T^k)$  for $(\xi_1, \xi_2, \ldots, \xi_T)$, realization of $(\xi_1^k, \ldots,\xi_T^k)$, 
\par \hspace*{1cm}i.e., sample a set of $T+1$ nodes $(n_0^k, n_1^k, n_2^k,\ldots, n_T^k)$ such that $n_0^k =  n_0$ is the root node, 
\par \hspace*{1cm}$n_1^k = n_1$ is the node corresponding to the first stage, and  for every $t=2,\ldots,T$, node 
\par \hspace*{1cm}$n_t^k$ is a child of node $n_{t-1}^k$.
\par \hspace*{1cm}{\textbf{For}} $t=1,\ldots,T$,
\par \hspace*{1.7cm}{\textbf{For}} every node $n$ of stage $t-1$,
\par \hspace*{2.4cm}{\textbf{For}} every child node $m$ of node $n$, 
\par \hspace*{3.1cm}solve \eqref{equationalgorithm3} and denote by $x_{m}^k$ be an optimal solution. 
\par \hspace*{2.4cm}{\textbf{End For}}
\par \hspace*{2.4cm}{\textbf{If}} $t \geq 2$ and $n \neq n_{t-1}^k$, compute 
\begin{equation}\label{nonewcutalg3}
\theta_n^k = \mathcal{Q}_n^{k-1}(x_{[n]}^k) \mbox{ and }\beta_n^k \in \partial \mathcal{Q}_n^{k-1}(x_{[n]}^k).
\end{equation}
\par \hspace*{2.4cm}{\textbf{Else if }}$t \geq 2$ and $n = n_{t-1}^k$, then for every $m \in C(n)$ compute a subgradient 
\par \hspace*{3.1cm}$\pi_{k, m}$ of $\mathfrak{Q}_{n}^{k-1}(\cdot, \xi_{m})$ at $x_{[n]}^k$:
\begin{equation}\label{pikmdef}
\begin{array}{lll}
\pi_{k, m}& = &f'_{\tau_m, x_{[n]}}\Big(x_{[n]}^k, x_{m}^{k}, \Psi_{m} \Big)+ g'_{\tau_m, x_{[n]}}\Big(x_{[n]}^k, x_{m}^{k}, \Psi_{m} \Big) \pi_{k, m, 1}\\
&&+ \left(\begin{array}{c}A_{1, m}^\transp\\\vdots\\A_{\tau_m -1, m}^\transp \end{array}   \right) \pi_{k, m, 2}+ [\beta_{m, 1}^{0}, \ldots, \beta_{m, 1}^{k-1}]\pi_{k, m, 3},
\end{array}
\end{equation}
\par \hspace*{3.1cm}where $f'_{\tau_m, x_{[n]}}\Big(x_{[n]}^k, x_{m}^{k}, \Psi_{m} \Big)$ is a 
subgradient of convex function 
\par \hspace*{3.1cm}$f_{\tau_m} ( \cdot, x_{m}^{k}, \Psi_{m} )$ at $x_{[n]}^k$ and the $i$-th column of matrix $g'_{\tau_m, x_{[n]}}(x_{[n]}^k, x_{m}^{k}, \Psi_{m})$ 
\par \hspace*{3.1cm}is a subgradient at $x_{[n]}^k$ of the $i$-th component of convex function $g'_{\tau_m, x_{[n]}}(\cdot, x_{m}^{k}, \Psi_{m})$.
\par \hspace*{3.1cm}Update $\theta_{n}^k$ and $\beta_{n}^k$ computing
\begin{equation}\label{formulathetankalg3}
\theta_{n}^k =  \sum_{m \in C(n)}\; p_{k, m} \Phi_m \mathfrak{Q}_n^{k-1}( x_{[n]}^k, \xi_{m})  \mbox{ and }\beta_{n}^k = \sum_{m \in C(n)}\; p_{k, m} \Phi_m \pi_{k, m}
\end{equation}
\hspace*{3.1cm}where $p_{k, m}$ satisfies:
$$
\theta_{n}^k = \displaystyle \sup_{p \in \mathcal{P}_{n}} \; \sum_{m \in C(n)} \; p_m \Phi_m  \mathfrak{Q}_n^{k-1}( x_{[n]}^k, \xi_{m})= \displaystyle  \sum_{m \in C(n)}\; p_{k, m} \Phi_m \mathfrak{Q}_n^{k-1}( x_{[n]}^k, \xi_{m}).
$$
\par \hspace*{2.4cm}{\textbf{End If}}
\par \hspace*{1.7cm}{\textbf{End For}}
\par \hspace*{1cm}{\textbf{End For}}
\par \hspace*{1cm}Set $\theta_{n}^k=0$ and $\beta_{n}^k=0$ for every leaf $n$.
\par {\textbf{End For}}\\
\rule{\linewidth}{1pt}
\begin{thm}[Convergence analysis of Algorithm 3] \label{convalg3}
Consider the sequence of random variables $(x_n^k)_{k \in \mathbb{N}^*}, n \in \mathcal{N}$
and random functions $(\mathcal{Q}_n^{k})_{k \in \mathbb{N}}, n \in \mathcal{N}$,
generated by Algorithm 3. Let Assumptions (H2) hold and assume that $(\xi_t)$
is  a discrete random process with a finite set of possible realizations at each stage.
Also assume that the samples generated along the iterations are independent:
introducing the binary random variables 
$y_n^k$ such that $y_n^k=1$ if node $n$ is selected at iteration $k$ and $0$ otherwise,
the random variables $(y_{n_T^k}^k)_{k \in \mathbb{N}^*}$ are independent.
Then,
\begin{itemize}
\item[(i)] almost surely, for any node $n \in \mathcal{N}\backslash\{n_0\}$, we have
$$
\lim_{k \rightarrow +\infty } \;\mathcal{Q}_n^k(x_{[n]}^k)- \mathcal{Q}_n( x_{[n]}^k )=0.
$$
\item[(ii)] Almost surely, we have
$$\lim_{k \rightarrow +\infty} \;\mathfrak{Q}_{n_0}^{k}(x_0, \xi_1)=\mathcal{Q}_{n_0}(x_0),$$ i.e.,
the optimal value of the approximate first stage problems converges to the optimal
value of the first stage problem. Moreover, almost surely, any accumulation point of the sequence
$(x_{n_1}^k)_{k \in \mathbb{N}^*}$ is an optimal solution of the first stage problem.
\end{itemize}
\end{thm}
\begin{proof} 
We provide the main steps of 
the proof which follows closely the proofs of Sections \ref{framework}  and \ref{sec:relativelcr}.

We prove (i) by backward induction on the number of stages. Following the proof of Proposition \ref{convexityrec}, we show
that $\mathcal{Q}_n$ is continuous on $\mathcal{X}_{1} \small{\times} \ldots \small{\times} \mathcal{X}_{\tau_n}$ for all 
$n \in \mathcal{N}\backslash\{n_0\}$. Following the proof of
Lemma \ref{lipqtk}, we show that for all $n \in \mathcal{N}\backslash\{n_0\}$ and $k$ sufficiently large,
say $k \geq T_0$, $\mathcal{Q}_n^k$ is Lipschitz continuous and for every $m$ the sequence
$(\pi_{k, m})_k$ given by \eqref{pikmdef} is bounded.

We also observe that 
\begin{equation}\label{thetankqnk}
\mathcal{Q}_n^{k}( x_{[n]}^k )= \theta_n^k, \forall k \geq 1,\;\forall n \in \mathcal{N}\backslash\{n_0\}.
\end{equation}
Indeed, if $n$ is a leaf the above relation holds by definition of $\theta_n^k$ and $\mathcal{Q}_n^{k}$.
Let us show \eqref{thetankqnk} when $n$ is not a leaf.
In this case, if $k \notin \mathcal{S}_n=\{k \geq 1 : n_{\tau_n}^k = n\}$ then
using \eqref{nonewcutalg3}, we have $\theta_n^k = \mathcal{Q}_n^{k-1}(x_{[n]}^k)$
and $\mathcal{Q}_n^{k}( x_{[n]}^k )=\max(\theta_n^k, \mathcal{Q}_n^{k-1}( x_{[n]}^k ))=\theta_n^k$.
To show \eqref{thetankqnk} for $k \in \mathcal{S}_n$, observe that
for a given node $n$, a new cut is added
at iteration $k$ for $\mathcal{Q}_n$ only when $k \in \mathcal{S}_n$.
It follows that
$$
\mathcal{Q}_n^{k}( x )=\max_{\ell \in \mathcal{S}_n^k} \theta_n^{\ell} + \langle  \beta_n^\ell, x - x_{[n]}^{\ell}\rangle
$$
where
$\mathcal{S}_n^k = \{1 \leq  j \leq k : j \in \mathcal{S}_n\} \bigcup \{0\}$.
For $k \in \mathcal{S}_n$ and $\ell \in \mathcal{S}_n^k$  with $1 \leq \ell<k$ we have 
$$
\begin{array}{lll}
\theta_n^k & = & \displaystyle \sup_{p \in \mathcal{P}_n} \; \sum_{m \in C(n)} \; p_m \Phi_m  \mathfrak{Q}_n^{k-1}( x_{[n]}^k, \xi_{m})\mbox{ using }\eqref{formulathetankalg3} \mbox{ and the fact that }k \in \mathcal{S}_n,\\
& \geq  & \displaystyle \sup_{p \in \mathcal{P}_n} \; \sum_{m \in C(n)} \; p_m \Phi_m  \mathfrak{Q}_n^{\ell-1}( x_{[n]}^k, \xi_{m})\mbox{ by monotonicity},\\
& \geq  & \displaystyle \sup_{p \in \mathcal{P}_n} \; \sum_{m \in C(n)} \; p_m \Phi_m  \Big( \mathfrak{Q}_n^{\ell-1}( x_{[n]}^{\ell}, \xi_{m}) +  \langle \pi_{\ell, m}, x_{[n]}^{k} - x_{[n]}^{\ell}\rangle   \Big) \mbox{ by definition of }\pi_{\ell, m},\\
& \geq  &\displaystyle  \sum_{m \in C(n)} \; p_{\ell, m} \Phi_m  \Big( \mathfrak{Q}_n^{\ell-1}( x_{[n]}^{\ell}, \xi_{m}) +  \langle \pi_{\ell, m}, x_{[n]}^{k} - x_{[n]}^{\ell}\rangle   \Big) \mbox{ since }(p_{\ell, m})_{m \in C(n)} \in \mathcal{P}_n,\\
& \geq  & \theta_n^{\ell} + \langle \beta_n^\ell, x_{[n]}^{k} - x_{[n]}^{\ell}   \rangle \mbox{ using }\eqref{formulathetankalg3} \mbox{ and the fact that }\ell \in \mathcal{S}_n.
\end{array}
$$
Observing that $k \in \mathcal{S}_n^k$, we get
$
\mathcal{Q}_n^{k}( x_{[n]}^k )=\max\Big(\theta_n^k, \max\Big(\theta_n^{\ell} + \langle  \beta_n^\ell, x_{[n]}^{k} - x_{[n]}^{\ell}\rangle, \ell<k, \ell \in \mathcal{S}_n^k \Big) \Big) = \theta_n^k
$
which shows \eqref{thetankqnk}.
We now prove (i) by induction (the proof is similar to the proof of (i) in Theorem \ref{convalg1}).

The induction hypothesis is that for each node $m$ of stage $t+1$, 
\begin{equation} \label{induchypothesis0}
\lim_{k \rightarrow +\infty} \;\mathcal{Q}_m( x_{[m]}^{k}  ) - \mathcal{Q}_m^{k}(x_{[m]}^{k})=0.
\end{equation}

The above relation is satisfied for every leaf $m$ of the tree.
Now assume that the induction hypothesis is true for each node $m$ of stage $t+1$ for some $t \in \{1,\ldots,T\}$.
We want to show that for each node $n$ of stage $t$, 
\begin{equation} \label{induchypothesis}
\lim_{k \rightarrow +\infty} \;\mathcal{Q}_n(x_{[n]}^{k}) - \mathcal{Q}_n^{k}(x_{[n]}^{k})=0.
\end{equation}
We first show that for each node $n$ of stage $t$, 
\begin{equation} \label{induchypothesisSn}
\lim_{k \rightarrow +\infty, \,k \in \mathcal{S}_n} \;\mathcal{Q}_n(x_{[n]}^{k}) - \mathcal{Q}_n^{k}(x_{[n]}^{k})=0.
\end{equation}
We deduce from the induction hypothesis \eqref{induchypothesis0} and Lemma A.1 in \cite{lecphilgirar12}
that
for each node $m$ of stage $t+1$
\begin{equation} \label{induchypothesisbis}
\lim_{k \rightarrow +\infty} \;\mathcal{Q}_m( x_{[m]}^{k}  ) - \mathcal{Q}_m^{k-1}(x_{[m]}^{k})=0.
\end{equation}
We then have for $k \in \mathcal{S}_n$
\begin{equation} \label{ineqproofintd0}
\begin{array}{lll}
\mathfrak{Q}_n^{k-1}\Big( x_{[n]}^k, \xi_m\Big) &=&  F_{\tau_m}\Big(x_{[m]}^k, \Psi_{m} \Big)-\mathcal{Q}_m\Big( x_{[m]}^k  \Big)+\mathcal{Q}_m^{k-1}\Big( x_{[m]}^k \Big) \\
 & \geq & \mathfrak{Q}_n \Big(x_{[n]}^k, \xi_{m} \Big)-\mathcal{Q}_m\Big( x_{[m]}^k  \Big)+\mathcal{Q}_m^{k-1}\Big( x_{[m]}^k \Big) 
\end{array}
\end{equation}
where the last inequality comes from the relation
$$
\mathfrak{Q}_n \Big(x_{[n]}^k, \xi_{m} \Big)=
\left\{
\begin{array}{l}
\displaystyle \inf_{x_m} \;F_{\tau_m}(x_{[n]}^k, x_m, \Psi_{m})\\
x_m \in X_{\tau_m}\Big( x_{[n]}^{k}, \xi_{m}  \Big)
\end{array}
\right\}
\leq F_{\tau_m}\Big(x_{[m]}^k, \Psi_{m} \Big)
$$
which holds since 
$x_m^k \in X_{\tau_m}\Big( x_{[n]}^{k}, \xi_{m}  \Big)$ for $k \in \mathcal{S}_n, m \in C(n)$.
It follows that
\begin{equation} \label{intdepnoeudm}
\begin{array}{lll}
0  &\leq &\mathfrak{Q}_n \Big(x_{[n]}^k, \xi_{m} \Big)-\mathfrak{Q}_n^{k-1} \Big(x_{[n]}^k, \xi_{m} \Big)  \leq  \mathcal{Q}_m \Big( x_{[m]}^k  \Big)-\mathcal{Q}_m^{k-1} \Big( x_{[m]}^k \Big).
\end{array}
\end{equation}
Combining \eqref{induchypothesisbis} and \eqref{intdepnoeudm} we obtain that for 
every node $m \in C(n)$
\begin{equation}\label{limitchildren}
\lim_{k \rightarrow +\infty, k \in \mathcal{S}_n} \;\mathfrak{Q}_n \Big(x_{[n]}^k, \xi_{m} \Big)-\mathfrak{Q}_n^{k-1} \Big(x_{[n]}^k, \xi_{m} \Big)=0.
\end{equation}
Next, for $k \in \mathcal{S}_n$,
$$
\begin{array}{l}
0 \leq \mathcal{Q}_n (x_{[n]}^k ) - \mathcal{Q}_n^k (x_{[n]}^k ) =
\displaystyle \sup_{p \in \mathcal{P}_n} \; \sum_{m \in C(n)} \; p_m \Phi_m  \mathfrak{Q}_n( x_{[n]}, \xi_{m}) - \theta_n^k \\
= \displaystyle \sup_{p \in \mathcal{P}_n} \; \sum_{m \in C(n)} \; p_m \Phi_m  \mathfrak{Q}_n( x_{[n]}, \xi_{m})-\displaystyle \sup_{p \in \mathcal{P}_{n}} \; \sum_{m \in C(n)} \; p_m \Phi_m  \mathfrak{Q}_n^{k-1}( x_{[n]}^k, \xi_{m})\\
\leq \displaystyle \sup_{p \in \mathcal{P}_n} \; \sum_{m \in C(n)} \; p_m \Phi_m  ( \mathfrak{Q}_n( x_{[n]}, \xi_{m})- \mathfrak{Q}_n^{k-1}( x_{[n]}^k, \xi_{m}) )\\
\end{array}
$$
Combining the above relation with \eqref{limitchildren} we have shown \eqref{induchypothesisSn}.

Next, following the end of the proof of Theorem \ref{convalg1}, we show by contradiction and using the Strong Law of Large Numbers
that 
$$
\displaystyle \lim_{k \rightarrow +\infty,\, k \notin \mathcal{S}_n  } \mathcal{Q}_{n}(x_{[n]}^{k})-\mathcal{Q}_{n}^{k-1}(x_{[n]}^{k} )=0,
$$
implying by monotonicity
\begin{equation}\label{inductionm5}
\displaystyle \lim_{k \rightarrow +\infty,\, k \notin {\tilde{\mathcal{S}}}_n  } \mathcal{Q}_{n}(x_{[n]}^{k})-\mathcal{Q}_{n}^{k}(x_{[n]}^{k} )=0,
\end{equation}
which achieves the proof of (ii).

Finally, the proof of (ii) is analogous to the proof of (ii) in Theorem \ref{convalg1}.\hfill
\end{proof}
We have an analogue of Remark \ref{importantremark} for Algorithm 3:
\begin{rem} \label{importantremark2}
Similarly to Algorithm 1, in Algorithm 3, decisions are computed at every iteration for all the nodes of the scenario tree.
However, in practice, decisions will only be computed for the nodes of the sampled scenarios and their children nodes.
This variant of Algorithm 3 will build the same cuts and compute the same decisions for the nodes of the
sampled scenarios as Algorithm 3. For this variant, for a node $n$, the decision variables $(x_n^k)_k$ are defined for
an infinite subset ${\tilde{\mathcal{S}}}_{n}$ of iterations where the sampled scenario passes through the parent node of node $n$, i.e., 
${\tilde{\mathcal{S}}}_{n}=\mathcal{S}_{\mathcal{P}(n)}$.
With this notation, applying Theorem \ref{convalg3}-(i), we get for $t=2,\ldots,T$,
$$
\lim_{k \rightarrow +\infty, k \in {\tilde{\mathcal{S}}}_{n}} \;\mathcal{Q}_n^k(x_{[n]}^k)- \mathcal{Q}_n( x_{[n]}^k )=0,
$$
almost surely, while Theorem \ref{convalg3}-(ii) still holds.
\end{rem}
We also have an analogue of Remark  \ref{remassumptions} for Algorithm 3.

\if{

Let us now fix $\varepsilon>0$ and $n \in {\tt{Nodes}}(t)$.
From the continuity of $\mathfrak{Q}_n(\cdot, \xi_m)$ and the fact that the sequence $(x_{[n]}^{k})_{k \in K_n}$ converges to $x_{[n]}^{*}$, 
there exists $k_{1}$ such that for $k \in K_n$ and $k \geq k_{1}$, we have for all $m \in C(n)$,
\begin{equation}\label{eqcontQb}
|\mathfrak{Q}_{n}(x_{[n]}^k, \xi_{m})-\mathfrak{Q}_{n}(x_{[n]}^{*}, \xi_{m})| \leq \frac{\varepsilon}{10}.
\end{equation}
Denoting by $0<M_1<+\infty$ an upper bound on $\|\pi_{k m}\|$ (valid for all $m$ and $k \geq T_0$),
since the sequence $(x_{[n]}^{k})_{k \in K_n}$ converges to $x_{[n]}^{*}$,
there exists $k_{2} \in K_n$ with $k_2 \geq T_0$ such that for $k \geq k_{2}$ and $k \in K_n$
\begin{equation}\label{eqfinb}
\|x_{[n]}^{k}-x_{[n]}^{*}\|  \leq \frac{\varepsilon}{10 M_1 }.
\end{equation}
Using \eqref{limitchildren}, for every $m \in C(n)$, there exists $k(3, m) \in K_{m}$
with $k(3, m) \geq \max(k_1, k_2)\geq T_0$
such that for every $k \in K_{m}$ with $k \geq k(3, m)$, we have
\begin{equation} \label{eqinductionb}
0 \leq \mathfrak{Q}_n\Big(x_{[n]}^k, \xi_{m} \Big)-\mathfrak{Q}_n^{k-1} \Big(x_{[n]}^k, \xi_{m} \Big) \leq \frac{\varepsilon}{10}.
\end{equation}
From the continuity of $\mathcal{Q}_n$ and the fact that the sequence $(x_{[n]}^{k})_{k \in K_n}$ converges to $x_{[n]}^{*}$,
there exists $k_{4}$ such that for $k \in K_n$ and $k \geq k_{4}$, we have
\begin{equation}\label{eqcontQbis}
|\mathcal{Q}_{n}(x_{[n]}^k)-\mathcal{Q}_{n}(x_{[n]}^{*})| \leq \frac{\varepsilon}{2}.
\end{equation}
Take now an arbitrary $k \in K_n$ with
$$
k \geq \max(T_0, k_1, k_2, \max(k(3,m),\;m \in C(n)), k_4)=\max(\max(k(3,m),\;m \in C(n)), k_4).
$$
Since $k \geq k(3, m)$ for every $m \in C(n)$, we have
\begin{equation} \label{eqfin0b}
\begin{array}{lll}
\theta_n^k &= &\mathcal{Q}_n^{k}( x_{[n]}^k ) =   \displaystyle \sup_{p \in \mathcal{P}_n} \; \sum_{m \in C(n)} \; p_m \Phi_m  \mathfrak{Q}_n^{k-1}( x_{[n]}^k, \xi_{m})\\
 & \geq  & \displaystyle \sup_{p \in \mathcal{P}_n} \;\sum_{m \in C(n)} \; p_m \Phi_m  \mathfrak{Q}_n^{k(3,m)-1}( x_{[n]}^k, \xi_{m})\\
& \geq  & \displaystyle \sup_{p \in \mathcal{P}_n} \; \sum_{m \in C(n)} \; p_m \Phi_m \left[ \mathfrak{Q}_{n}^{k(3,m)-1}(x_{[n]}^{k(3,m)}, \xi_{m}) +\langle  \pi_{k(3,m) \;m}, x_{[n]}^k - x_{[n]}^{k(3,m)} \rangle \right]
\end{array}
\end{equation}
using the convexity of $\mathfrak{Q}_{n}^{k(3,m)-1}(\cdot, \xi_m)$ and the fact that
$\pi_{k(3,m) \;m}$ is a subgradient of
$\mathfrak{Q}_{n}^{k(3,m)-1}(\cdot, \xi_{m})$ at $x_{[n]}^{k(3,m)}$ (recall that $k(3,m) \in K_{m}$).
We obtain
\begin{equation} \label{eqfin1b}
\begin{array}{l}
0 \leq \mathcal{Q}_n(x_{[n]}^k)-\mathcal{Q}_n^k(x_{[n]}^k) \leq \\
\displaystyle \sup_{p \in \mathcal{P}_n} \;\displaystyle \sum_{m \in C(n)} p_m\Phi_{m}\left[\mathfrak{Q}_n\Big(x_{[n]}^k, \xi_m \Big)-\mathfrak{Q}_{n}^{k(3,m)-1}(x_{[n]}^{k(3,m)}, \xi_{m})- \langle  \pi_{k(3,m)\,m}, x_{[n]}^k - x_{[n]}^{k(3,m)} \rangle \right]  \\
\leq \displaystyle \sup_{p \in \mathcal{P}_n} \;\displaystyle \sum_{m \in C(n)} p_m\Phi_{m}  \Big| \mathfrak{Q}_n\Big(x_{[n]}^k, \xi_m \Big)-\mathfrak{Q}_n\Big(x_{[n]}^*, \xi_m \Big)\Big|\\ 
\;\;\; + \displaystyle \sup_{p \in \mathcal{P}_n} \;\displaystyle \sum_{m \in C(n)} p_m\Phi_{m}  \Big|\mathfrak{Q}_{n}(x_{[n]}^{k(3,m)}, \xi_{m})-\mathfrak{Q}_{n}^{k(3,m)-1}(x_{[n]}^{k(3,m)}, \xi_{m})\Big|\\ 
 \;\;\; + \displaystyle \sup_{p \in \mathcal{P}_n} \;\displaystyle \sum_{m \in C(n)} p_m\Phi_{m}   \Big|\mathfrak{Q}_{n}(x_{[n]}^{k(3,m)}, \xi_{m})-\mathfrak{Q}_{n}(x_{[n]}^{*}, \xi_{m})\Big|\\
\;\; +\displaystyle \sup_{p \in \mathcal{P}_n} \;\displaystyle \sum_{m \in C(n)} p_m\Phi_{m} \|\pi_{k(3,m)\,m}\| \Big(  \|x_{[n]}^k - x_{[n]}^{*}\| + \|x_{[n]}^{k(3,m)}-x_{[n]}^{*}\| \Big) \leq \frac{\varepsilon}{2}.
\end{array}
\end{equation}
The last inequality was obtained using relations 
\eqref{eqcontQb}, \eqref{eqfinb}, \eqref{eqinductionb}, and the fact that $\sum_{m \in C(n)} p_m \Phi_{m}=1$ for any $p \in \mathcal{P}_n$.
Combining the above relation with \eqref{eqcontQbis}, we have shown that for every $\varepsilon>0$, for every
$k \geq \max(k_1, k_2, \max(k(3, m),\;m \in C(n)), k_4)$ with $k \in K_n$ we have
$|\mathcal{Q}_n^{k}(x_{[n]}^k)-\mathcal{Q}_n(x_{[n]}^{*})|\leq \varepsilon$. We have thus shown \eqref{induchypothesis}, which achieves the induction step
and the proof of (ii).

\begin{prop}\label{propconv1} 
Consider the sequence of random variables $x_t^k$ 
generated by Algorithm 1.
Let Assumptions (H1), (H2), and (H3) hold. Then with probability one, there exist 
$x_n^*, n \in \mathcal{N}$, and infinite subsets $K_n, n \in \mathcal{N}$, of integers, such that
for $t=1,\ldots,T$,
$$
H_1(t): \;\;\;\forall n \in {\tt{Nodes}}(t), \;\;\lim_{k \rightarrow +\infty,\, k \in K_n} (x_1^k, \ldots, x_{t}^{k}) = x_{[n]}^{*} \in  
\mathcal{X}_1 \small{\times} \ldots \small{\times} \mathcal{X}_t
$$
where the $i$-th component of $x_{[n]}^{*}$ is $x_{\mathcal{P}^{t-i}(n)}^*$ for $i=1,\ldots,t$. 
\end{prop}
\begin{proof}
We show $H_1(t)$ for $t=1,\ldots,T$, by induction on $t$.
Since the sequence $(x_{1}^{k})_{k \in \mathbb{N}^*}$ is a sequence of the compact set $\mathcal{X}_1$,
there exists an infinite set $K_1$ such that the sequence $(x_{1}^{k})_{k \in K_1}$ converges
to some $x_1^*$. This shows $H_1(1)$. Assume now that $H_1(t)$ holds for some $t \in \{1,\ldots,T-1\}$
and let us show that $H_1(t+1)$ holds. 
Let us take a node $m_0$ of stage $t+1$ and consider its parent node $n$.
Let us partition
$K_n$ as $K_n=\displaystyle{\cup_{m \in C(n)}} K'_m$ where
\begin{equation} \label{defkpm}
K'_m=\{k \in K_n : (\tilde \xi_1^k,\tilde \xi_2^k, \ldots,\tilde \xi_{t+1}^k)=\xi_{[m]}\}.
\end{equation}

Note that for every $m \in C(n)$, the set $K'_m$ has an infinite number of elements.
Indeed, since $K_n$ is infinite, the set of samples
$(\tilde \xi_1^k, \ldots, \tilde \xi_T^k)_{k \in K_n}$ constitute an infinite set of samples of $(\xi_1, \ldots, \xi_T)$
that all pass through node $n$. Due to Assumption (H3) and the Borel-Cantelli lemma, for every child
node $m$ of node $n$, the set of samples with indices in $K_n$
that pass through $m$ is infinite and this set is $K'_m$.
In particular, since the sequence $(x_{t+1}^k)_{k \in K'_{m_0}}$ is an infinite sequence from the compact set
$\mathcal{X}_{t+1}$, there exists $x_{m_0}^* \in \mathcal{X}_{t+1}$ and
an infinite subset $K_{m_0} \subset K'_{m_0} \subset K_n= K_{\mathcal{P}(m_0)}$ such
that 
\begin{equation} \label{firstconvoptsol}
\lim_{k \rightarrow +\infty,\, k \in K_{m_0}} x_{t+1}^k = x_{m_0}^*.
\end{equation}
Since node $n$ belongs to stage $t$, using the induction hypothesis, we have
$$
\lim_{k \rightarrow +\infty,\, k \in K_n} (x_1^k, \ldots, x_{t}^{k}) = x_{[n]}^{*} \in  
\mathcal{X}_1 \small{\times} \ldots \small{\times} \mathcal{X}_t.
$$
Since $K_{m_0}$ is an infinite set contained in $K_n$, we also have
$$
\lim_{k \rightarrow +\infty,\, k \in K_{m_0}} (x_1^k, \ldots, x_{t}^{k}) = x_{[n]}^{*} \in  
\mathcal{X}_1 \small{\times} \ldots \small{\times} \mathcal{X}_t,
$$
which, combined with \eqref{firstconvoptsol}, implies
$$
\lim_{k \rightarrow +\infty,\, k \in K_{m_0}} (x_1^k, \ldots, x_{t}^{k}, x_{t+1}^{k}) = (x_{[n]}^{*},x_{m_0}^*)=x_{[m_0]}^{*} \in  
\mathcal{X}_1 \small{\times} \ldots \small{\times} \mathcal{X}_{t+1}
$$
(since $n$ is the parent node of $m_0$) and achieves the induction step.\hfill
\end{proof}
For convenience, we shall denote by $(y(m,k))_{k \geq 0}$ the sequence of iterations, sorted in ascending order, belonging to $K_m$.
The following lemma will be useful in the sequel:
\begin{lemma}\label{convkm1}
Assume that for some $t=1,\ldots,T,$ and some $m \in {\tt{Nodes}}(t)$, we have
\begin{equation}\label{firstconvkm1}
\lim_{k \rightarrow +\infty,\,k \in K_m} \; \mathcal{Q}_{t+1}(x_{1:t}^{k})-\mathcal{Q}_{t+1}^{k}(x_{1:t}^{k})=0
\end{equation}
where sets $K_m$ are defined in the proof of Proposition \ref{propconv1}.
Then 
\begin{equation}\label{secconvkm1}
\lim_{k \rightarrow +\infty,\,k \in K_m} \; \mathcal{Q}_{t+1}^{k-1}(x_{1:t}^{k})=\mathcal{Q}_{t+1}(x_{[m]}^{*})
\end{equation}
where $x_{[m]}^{*}$ is defined in the proof of Proposition \ref{propconv1}.
\end{lemma}
\begin{proof}
For some $t \in \{1,\ldots,T\}$ and $m \in {\tt{Nodes}}(t)$, assume that 
\eqref{firstconvkm1} holds. 
We want to show that \eqref{secconvkm1} holds. Fix $\varepsilon>0$.
Let $L$ be the Lipschitz constant for functions $\mathcal{Q}_{t+1}^k, k \geq T$, given by
Lemma \ref{lipforqtk}. From Proposition \ref{propconv1}, we have $\lim_{k \rightarrow +\infty} \;x_{1:t}^{y(m,k)}=x_{[m]}^{*}$.
It follows that there exists $k(m,0)$ such that for $k \geq k(m,0)$, we have
\begin{equation} \label{conxtmL}
\|x_{1:t}^{y(m,k)} - x_{[m]}^{*}  \| \leq \frac{\varepsilon}{4L}.
\end{equation}
Next, since $\mathcal{Q}_{t+1}$ is continuous on $\mathcal{X}_1 \small{\times} \ldots \small{\times} \mathcal{X}_t$ with 
$x_{1:t}^{y(m,k)}, x_{[m]}^{*}  \in \mathcal{X}_1 \small{\times} \ldots \small{\times} \mathcal{X}_t$,
\begin{equation}\label{induc0}
\lim_{k \rightarrow +\infty} \; \mathcal{Q}_{t+1}(x_{1:t}^{y(m,k)})=\mathcal{Q}_{t+1}(x_{[m]}^{*}).
\end{equation}
Using \eqref{firstconvkm1} and \eqref{induc0}, we have that
\begin{equation}\label{induc1}
\lim_{k \rightarrow +\infty} \; \mathcal{Q}_{t+1}^{y(m,k)}(x_{1:t}^{y(m,k)})=\mathcal{Q}_{t+1}(x_{[m]}^{*}).
\end{equation}
Using \eqref{induc0} and \eqref{induc1}, there exists $k(m,1)$ satisfying $k(m,1) \geq k(m,0)$ 
and $y(m, k(m,1)) \geq T$ such that for $k \geq k(m,1)$,
\begin{equation} \label{km1km0}
\begin{array}{l}
|\mathcal{Q}_{t+1}(x_{1:t}^{y(m,k)})-\mathcal{Q}_{t+1}(x_{[m]}^{*})| \leq \varepsilon,\;\;\;|\mathcal{Q}_{t+1}^{y(m,k)}(x_{1:t}^{y(m,k)})-\mathcal{Q}_{t+1}(x_{[m]}^{*})|\leq \frac{\varepsilon}{2}.
\end{array}
\end{equation}
Then for any $k \geq k(m,1)+1$, we deduce that
\begin{equation}\label{proofrhszero}
\begin{array}{l}
 \mathcal{Q}_{t+1}^{y(m, k(m,1))}(x_{1:t}^{y(m,k)}) - \mathcal{Q}_{t+1}(x_{[m]}^{*})
\leq \mathcal{Q}_{t+1}^{y(m,k)-1}(x_{1:t}^{y(m,k)}) - \mathcal{Q}_{t+1}(x_{[m]}^{*}) \\
\leq \mathcal{Q}_{t+1}(x_{1:t}^{y(m,k)}) - \mathcal{Q}_{t+1}(x_{[m]}^{*}) \leq \varepsilon.
\end{array}
\end{equation}
In \eqref{proofrhszero}, the first inequality is due to $\mathcal{Q}_{t+1}^{y(m, k(m, 1))} \leq \mathcal{Q}_{t+1}^{y(m, k)-1}$
because $y(m, k)-1 \geq y(m, k(m, 1))$ and the second inequality to $\mathcal{Q}_{t+1}^{y(m, k)-1} \leq \mathcal{Q}_{t+1}$.

We deduce that for $k \geq k(m, 1)+1$, we have
$$
\begin{array}{l}
|\mathcal{Q}_{t+1}^{y(m, k(m, 1))}(x_{1:t}^{y(m, k)})- \mathcal{Q}_{t+1}(x_{[m]}^{*})| \\
\leq  
|\mathcal{Q}_{t+1}^{y(m, k(m, 1))}(x_{1:t}^{y(m, k(m, 1))}) - \mathcal{Q}_{t+1}^{y(m, k(m, 1))}(x_{1:t}^{y(m, k)})| \\
\;\;\;+ |\mathcal{Q}_{t+1}^{y(m, k(m, 1))}(x_{1:t}^{y(m, k(m, 1))}) - \mathcal{Q}_{t+1}(x_{[m]}^{*}) | \\
\leq \frac{\varepsilon}{2} + L \|x_{1:t}^{y(m, k)}-x_{1:t}^{y(m, k(m, 1))}\| \leq \varepsilon
\end{array}
$$
using \eqref{km1km0}, the fact that $\mathcal{Q}_{t+1}^{y(m, k(m, 1))}$ is L-Lipschitz, and \eqref{conxtmL}.
From \eqref{proofrhszero} and the above relation, we get
$|\mathcal{Q}_{t+1}^{y(m, k)-1}(x_{1:t}^{y(m, k)}) - \mathcal{Q}_{t+1}(x_{[m]}^{*})| \leq \varepsilon$ for $k \geq k(m,1)+1$
which achieves the proof.\hfill
\end{proof}

Theorem \ref{convalg1}  shows the convergence of the sequence $\mathfrak{Q}_{1}^k(x_0, \xi_1)$ 
to $\mathcal{Q}_{1}(x_0)$ and that 
any accumulation point of the sequence $(x_1^k)_{k \in \mathbb{N}}$ is an optimal solution 
of the first stage problem \eqref{firststagepb}.
\begin{thm}[Convergence analysis of Algorithm 1 - Interstage independent process]\label{convalg1} 
Consider the sequences of random variables $x_t^k$ and random functions $\mathcal{Q}_ t^k$
generated by Algorithm 1.
Let Assumptions (H1), (H2), and (H3) hold. Then there exist 
$x_n^*, n \in \mathcal{N}$, and infinite subsets $K_n, n \in \mathcal{N}$, of integers,
defined in the proof of Proposition \ref{propconv1}, such that a.s.
\begin{itemize}
\item[(i)] for $t=2,\ldots,T$,
$$
H_1(t): \;\;\;\forall n \in {\tt{Nodes}}(t), \;\;\lim_{k \rightarrow +\infty,\, k \in K_n} (x_1^k, \ldots, x_{t}^{k}) = x_{[n]}^{*} \in  
\mathcal{X}_1 \small{\times} \ldots \small{\times} \mathcal{X}_t
 $$
where the $i$-th component of $x_{[n]}^{*}$ is $x_{\mathcal{P}^{t-i}(n)}^*$ for $i=1,\ldots,t$. 
\item[(ii)] For $t=2,\ldots,T$,
$$
H_2(t): \;\;\;\forall n \in {\tt{Nodes}}(t-1), \;\; \displaystyle \lim_{k \rightarrow +\infty,\, k \in K_n} \mathcal{Q}_{t}^{k}(x_{1:t-1}^{k})=\mathcal{Q}_{t}(x_{[n]}^{*} ) \mbox{ a.s. }
$$
\item[(iii)] $\displaystyle \lim_{k \rightarrow +\infty} \mathfrak{Q}_1^{k}(x_{0}, \xi_1)=\mathcal{Q}_{1}(x_0)$ and
if $f_1(\cdot, \Psi_1)$ is continuous on $\mathcal{X}_1$,
any accumulation point of the sequence $(x_1^k)_{k \in \mathbb{N}}$ is an optimal solution of the first stage problem \eqref{firststagepb}.
\end{itemize}
\end{thm}
\begin{proof}
Item (i) is Proposition \ref{propconv1}.

We show (ii) by induction backwards in time. 
Recalling that $\mathcal{Q}_{T+1}=\mathcal{Q}_{T+1}^{k}=0$, we have
$\mathfrak{Q}_T^{k-1}(x_{1:T-1}^k, \cdot)=\mathfrak{Q}_T(x_{1:T-1}^k, \cdot)$ for all $k \in \mathbb{N}$.
It follows that for all $k \in \mathbb{N}$,
\begin{equation}\label{h2T}
\begin{array}{ccl}
\mathcal{Q}_T^k(x_{1:T-1}^k)& \stackrel{Lemma \;  \ref{eqqeqthetlemma}}{=} &\theta_T^k \stackrel{\eqref{formulathetak}}{=} 
\rho_T(\mathfrak{Q}_T^{k-1}(x_{1:T-1}^k, \xi_T))\\
&= &\rho_T(\mathfrak{Q}_T(x_{1:T-1}^k, \xi_T)) \stackrel{\eqref{definitionQt}}{=} \mathcal{Q}_T(x_{1:T-1}^k).
\end{array}
\end{equation}
From (i), for every node $n\in {\tt{Nodes}}(T-1)$, we have $\lim_{k \rightarrow +\infty,\, k \in K_{n}} x_{1:T-1}^k = x_{[n]}^* \in \mathcal{X}_1 \small{\times} \ldots \small{\times} \mathcal{X}_{T-1}$.
From Proposition \ref{convexityrec}, $\mathcal{Q}_T$ is continuous on 
$\mathcal{X}_1 \small{\times} \ldots \small{\times} \mathcal{X}_{T-1}$. It follows that
for every node $n\in {\tt{Nodes}}(T-1)$, taking the limit when  $k \rightarrow +\infty,\, k \in K_{n}$ in \eqref{h2T}, we obtain
$$
\begin{array}{lll} 
\lim_{k \rightarrow +\infty,\, k \in K_{n}} \mathcal{Q}_T^{k}(x_{1:T-1}^k) = \lim_{k \rightarrow +\infty,\, k \in K_{n}} \mathcal{Q}_T(x_{1:T-1}^k) =  
\mathcal{Q}_T(x_{[n]}^*),
\end{array}
$$
which is $H_2(T)$.

Now assume that $H_2(t+1)$ holds for some $t \in \{2,\ldots,T-1\}$. We want to show that $H_2(t)$ holds.
Take a node $n \in {\tt{Nodes}}(t-1)$, consider an arbitrary children node $m$ of $n$,
and $k \in K_n$ with $K_n$ defined in the proof of Proposition \ref{propconv1}.
For all $k \in K_m$, since $K_m \subset K'_m \subset K_n$ with $K'_m$ given by 
\eqref{defkpm}, we have $(\tilde \xi_1^k, \ldots, \tilde \xi_2^k, \ldots, \tilde \xi_{t}^k)=\xi_{[m]}$.
In particular, $\tilde \xi_{t}^k=\xi_m$
(for the iterations $k \in K_m$, the sampled scenarios pass through node $m$).
We obtain
\begin{equation} \label{eqinit}
\begin{array}{llll}
\mathfrak{Q}_{t}^{k-1}(x_{1:t-1}^{k}, \xi_{m}) & = & \mathfrak{Q}_{t}^{k-1}(x_{1:t-1}^{k}, \tilde \xi_{t}^{k})  & \mbox{since }k \in K_m,\\
& = & f_t (x_{1:t}^{k}, \Psi_{m}) + \mathcal{Q}_{t+1}^{k-1}(x_{1:t}^k) & \mbox{by definition of }x_{t}^k,\\
& = & F_t( x_{1:t}^{k}, \Psi_{m}) - \mathcal{Q}_{t+1}(x_{1:t}^k) + \mathcal{Q}_{t+1}^{k-1}(x_{1:t}^k) & \mbox{by definition of }F_t.
\end{array}
\end{equation}
Next, since $x_t^k \in X_t(x_{0:t-1}^k, \tilde \xi_t^k)$, we have
\begin{equation} \label{ineqqt34}
\mathfrak{Q}_t(x_{1:t-1}^{k}, \xi_{m})=\mathfrak{Q}_t(x_{1:t-1}^{k}, \tilde \xi_{t}^{k})=
\left\{
\begin{array}{l}
\displaystyle \inf_{x_t}\;F_t(x_{1:t-1}^{k}, x_t, \Psi_{m})\\
x_t \in X_t(x_{0:t-1}^k, \tilde \xi_t^k) 
\end{array}
\right.
\leq F_t(x_{1:t}^k, \Psi_{m}).
\end{equation}
Plugging the above inequality into \eqref{eqinit}, we obtain for $k \in K_m$
\begin{equation} \label{firststep0}
\begin{array}{lll}
0 \leq \mathfrak{Q}_{t}(x_{1:t-1}^{k}, \xi_{m})-\mathfrak{Q}_{t}^{k-1}(x_{1:t-1}^{k}, \xi_{m}) &\leq & \mathcal{Q}_{t+1}(x_{1:t}^{k})-\mathcal{Q}_{t+1}^{k-1}(x_{1:t}^{k}).
\end{array}
\end{equation}
Since $m \in {\tt{Nodes}}(t)$, the induction hypothesis, (i), and the continuity of $\mathcal{Q}_{t+1}$ gives
$$
\lim_{k \rightarrow +\infty,\,k \in K_m} \; \mathcal{Q}_{t+1}(x_{1:t}^{k})-\mathcal{Q}_{t+1}^{k}(x_{1:t}^{k})=0.
$$
As a result, using Lemma \ref{convkm1}, (i), and the continuity of $\mathcal{Q}_{t+1}$, we get
$$
\lim_{k \rightarrow +\infty,\,k \in K_m} \; \mathcal{Q}_{t+1}(x_{1:t}^{k})-\mathcal{Q}_{t+1}^{k-1}(x_{1:t}^{k})=0.
$$
Combined with \eqref{firststep0}, we have shown that
\begin{equation}\label{limitchildren0}
\lim_{k \rightarrow +\infty,\,k \in K_m} \;\mathfrak{Q}_{t}(x_{1:t-1}^{k}, \xi_{m})-\mathfrak{Q}_{t}^{k-1}(x_{1:t-1}^{k}, \xi_{m})=0.
\end{equation}
Let us now fix $\varepsilon>0$.
We have shown in the proof of Proposition \ref{convexityrec} that for every
$m \in C(n)$ (with $|C(n)|=M$ finite), the function $\mathfrak{Q}_t(\cdot, \xi_m)$ is continuous
on $\mathcal{X}_1  {\small{\times}} \ldots {\small{\times}} \mathcal{X}_{t-1}$.
Since the sequence $(x_{1:t-1}^{k})_{k \in K_n}$ with $x_{1:t-1}^{k} \in \mathcal{X}_1  {\small{\times}} \ldots {\small{\times}} \mathcal{X}_{t-1}$ 
converges to $x_{[n]}^{*}$, 
there exists $k_{1}$ such that for all $k \in K_n$ and $k \geq k_{1}$, we have for all $m \in C(n)$,
\begin{equation}\label{eqcontQ}
|\mathfrak{Q}_{t}(x_{1:t-1}^k, \xi_{m})-\mathfrak{Q}_{t}(x_{[n]}^{*}, \xi_{m})| \leq \frac{\varepsilon}{10}.
\end{equation}
Now recall that in Lemma \ref{lipqtk}, we obtained the upper bound $M_0$
given by  \eqref{uppboundpitkj} for $\|\pi_{t, k , j}\|$, valid for every $t=2, \ldots,T, j=1,\ldots,M$, and
$k \geq T$. We can assume without loss of generality that $M_0>0$ (if $M_0=0$, it can be replaced by, say, $M_0+1>0$).
Since the sequence $(x_{1:t-1}^{k})_{k \in K_n}$ converges to $x_{[n]}^{*}$,
there exists $k_{2} \in K_n$ with $k_2 \geq T$ such that for $k \geq k_{2}$, we have for every $k \in K_n$
\begin{equation}\label{eqfin}
\|x_{1:t-1}^{k}-x_{[n]}^{*}\|  \leq \frac{\varepsilon}{10 M_0 }.
\end{equation}
Using \eqref{limitchildren0}, for every $m \in C(n)$, there exists $k(3, m) \in K_m$
with $k(3, m) \geq \max(k_1, k_2, T)$
such that for every $k \in K_m$ satisfying $k \geq k(3, m)$, we have
\begin{equation} \label{eqinduction}
0 \leq \mathfrak{Q}_{t}(x_{1:t-1}^{k}, \xi_{m})-\mathfrak{Q}_{t}^{k-1}(x_{1:t-1}^{k}, \xi_{m}) \leq \frac{\varepsilon}{10}.
\end{equation}
From the continuity of $\mathcal{Q}_t$ and the fact that the sequence $(x_{1:t-1}^{k})_{k \in K_n}$ converges to $x_{[n]}^{*}$,
there exists $k_{4}$ such that for $k \in K_n$ and $k \geq k_{4}$, we have
\begin{equation}\label{eqcontQ2}
|\mathcal{Q}_{t}(x_{1:t-1}^k)-\mathcal{Q}_{t}(x_{[n]}^{*})| \leq \frac{\varepsilon}{2}.
\end{equation}
Take now an arbitrary 
$$k \geq \max(k_1, k_2, \max(k(3, m), m \in C(n)), k_4)=\max(\max(k(3, m),m \in C(n)), k_4)$$ with $k \in K_n$.
Let $\tt{Index}(m)$ be such that $\xi_{t \tt{Index}(m)}=\xi_m$. 
Since $k \geq k(3, m)$ for every $m \in C(n)$, 
using the convexity of $\mathfrak{Q}_{t}^{k(3, m)-1}(\cdot, \xi_m)$ and the fact that
$\pi_{t k(3, m) \tt{Index}(m)}$ is a subgradient of convex function
$\mathfrak{Q}_{t}^{k(3, m)-1}(\cdot,\xi_{m})$ at $x_{1:t-1}^{k(3, m)}$,
we have 
\begin{equation} \label{eqfin0}
\begin{array}{l}
\theta_t^k = \mathcal{Q}_t^k(x_{1:t-1}^k)  =  \displaystyle \sup_{p \in \mathcal{P}_t} \;\displaystyle \sum_{m \in C(n)} p_m \Phi_{t {\tt{Index}(m)}} \mathfrak{Q}_{t}^{k-1}(x_{1:t-1}^k, \xi_{m})\\
 \geq   \displaystyle \sup_{p \in \mathcal{P}_t} \;\displaystyle \sum_{m \in C(n)} p_m \Phi_{t {\tt{Index}(m)}} \mathfrak{Q}_{t}^{k(3, m)-1}(x_{1:t-1}^k, \xi_{m})\\
\geq   \displaystyle \sup_{p \in \mathcal{P}_t} \;\displaystyle \sum_{m \in C(n)} p_m \Phi_{t {\tt{Index}(m)}} \left[\mathfrak{Q}_{t}^{k(3, m)-1}(x_{1:t-1}^{k(3, m)}, \xi_{m})  +  \langle \pi_{t k(3, m) \tt{Index}(m)}, x_{1:t-1}^k - x_{1:t-1}^{k(3, m)} \rangle\right].
\end{array}
\end{equation}
This gives
\begin{equation} \label{eqfin1}
\begin{array}{l}
0 \leq \mathcal{Q}_t(x_{1:t-1}^k)-\mathcal{Q}_t^k(x_{1:t-1}^k) \leq \\
\displaystyle \sup_{p \in \mathcal{P}_t} \;\displaystyle \sum_{m \in C(n)} p_m \Phi_{t {\tt{Index}(m)}} \left[ \mathfrak{Q}_t(x_{1:t-1}^k, \xi_{m})-\mathfrak{Q}_t^{k(3, m)-1}(x_{1:t-1}^{k(3,m)}, \xi_{m}) \right] \\
+\displaystyle \sup_{p \in \mathcal{P}_t} \;\displaystyle \sum_{m \in C(n)} p_m \Phi_{t {\tt{Index}(m)}}  \langle \pi_{t k(3, m) \tt{Index}(m)}, x_{1:t-1}^{k(3,m)} - x_{1:t-1}^k \rangle \\
\leq \displaystyle \sup_{p \in \mathcal{P}_t} \;\displaystyle \sum_{m \in C(n)} p_m \Phi_{t {\tt{Index}(m)}} \Big|\mathfrak{Q}_t(x_{1:t-1}^{k}, \xi_{m})-\mathfrak{Q}_t(x_{[n]}^{*}, \xi_{m})\Big|\\
\;\;\; + \displaystyle \sup_{p \in \mathcal{P}_t} \;\displaystyle \sum_{m \in C(n)} p_m \Phi_{t {\tt{Index}(m)}} \Big|\mathfrak{Q}_t(x_{1:t-1}^{k(3, m)}, \xi_{m})-\mathfrak{Q}_t^{k(3,m)-1}(x_{1:t-1}^{k(3, m)}, \xi_{m})\Big|\\
 \;\;\; + \displaystyle \sup_{p \in \mathcal{P}_t} \;\displaystyle \sum_{m \in C(n)} p_m \Phi_{t {\tt{Index}(m)}} \Big|\mathfrak{Q}_t(x_{1:t-1}^{k(3,m)}, \xi_{m})-\mathfrak{Q}_t(x_{[n]}^{*}, \xi_{m})\Big|\\
\;\; +\displaystyle \sup_{p \in \mathcal{P}_t} \;\displaystyle \sum_{m \in C(n)} p_m \Phi_{t {\tt{Index}(m)}} \left[ \|\pi_{t k(3,m) \tt{Index}(m)}\| \Big( \|x_{1:t-1}^k - x_{[n]}^{*}\| + \|x_{1:t-1}^{k(3,m)}-x_{[n]}^*{}\| \Big)\right]\\
\leq \frac{\varepsilon}{2}.
\end{array}
\end{equation}
The last inequality was obtained using relations 
\eqref{eqcontQ}, \eqref{eqfin}, \eqref{eqinduction}, 
 and the fact that $\sum_{m \in C(n)} p_m \Phi_{t {\tt{Index}(m)}}=1$ for any $p \in \mathcal{P}_t$.
In particular, \eqref{eqcontQ} (resp. \eqref{eqfin}) was used to bound from above 
 the third (resp fourth) of the four terms of the right-hand side of the penultimate
 inequality. This is possible since
 \eqref{eqcontQ} (resp. \eqref{eqfin}) which holds for $k \in K_n$ with $k \geq k_1$ (resp. $k \in K_n$ with $k \geq k_2$) was used with
 $k=k(3,m) \geq \max(k_1, k_2)\geq k_1$ satisfying $k(3,m) \in K_m \subset K_n$ 
 (resp. $k=k(3,m) \geq \max(k_1, k_2)\geq k_2$ satisfying $k(3,m) \in K_m \subset K_n$).
  
Combining the above relation \eqref{eqfin1} with \eqref{eqcontQ2}, we have shown that for every $\varepsilon>0$, for every
$k \geq \max(\max(k(3, m), m \in C(n)), k_4)$ with $k \in K_n$ we have
$|\mathcal{Q}_t^{k}(x_{1:t-1}^k)-\mathcal{Q}_t(x_{[n]}^{*})|\leq \varepsilon$. 
Since the node  $n$ was arbitrarily chosen in ${\tt{Nodes}}(t-1)$, we have shown $H(t)$, which achieves the induction step
and the proof of (ii).

(iii) By definition of $\mathfrak{Q}_1^{k-1}$ (see Algorithm 1), we have
$$
\mathfrak{Q}_1^{k-1}(x_{0}, \xi_1)=F_1(x_1^k, \Psi_1)-\mathcal{Q}_2(x_1^k)+\mathcal{Q}_2^{k-1}(x_1^k)
$$
which implies 
\begin{equation}\label{finaliii}
0 \leq \mathfrak{Q}_1(x_{0}, \xi_1)-\mathfrak{Q}_1^{k-1}(x_{0}, \xi_1)\leq \mathcal{Q}_2(x_1^k)-\mathcal{Q}_2^{k-1}(x_1^k).
\end{equation}
From (ii), $H_2(2)$ holds, which, combined with the continuity
of $\mathcal{Q}_2$, gives 
$$\lim_{k \rightarrow+\infty,\,k\in K_{n_1}}\; \mathcal{Q}_{2}^{k}(x_1^k)-\mathcal{Q}_{2}(x_1^k)=0
$$
where $n_1$ is the unique child of the root node (remember that $\xi_1$ is deterministic), i.e., the node associated to the first stage.
Using Lemma \ref{convkm1} with $t=1$ and the continuity of $\mathcal{Q}_{2}$, we obtain that
$\lim_{k \rightarrow+\infty,\,k\in K_{n_1}}\; \mathcal{Q}_{2}(x_1^k)-\mathcal{Q}_{2}^{k-1}(x_1^k)=0$.
Combined with \eqref{finaliii}, we have shown that
$\lim_{k \rightarrow+\infty,\,k\in K_{n_1}}\; \mathfrak{Q}_{1}^{k-1}( x_0, \xi_1  )=\mathfrak{Q}_{1}(x_0, \xi_1)=\mathcal{Q}_1(x_0)$.

It follows that for every $\varepsilon>0$, there exists $k_0 \in K_{n_1}$ such that for every $k \in K_{n_1}$ with
$k \geq k_0$, we have $0 \leq \mathcal{Q}_1(x_0)-\mathfrak{Q}_1^{k-1}(x_0, \xi_1)\leq\varepsilon$.
As a result, for every $k \in \mathbb{N}$ with $k \geq k_0-1$, we have
$$
0 \leq \mathcal{Q}_1(x_0)-\mathfrak{Q}_1^{k}(x_0, \xi_1) \leq \mathcal{Q}_1(x_0)-\mathfrak{Q}_1^{k_0-1}(x_0, \xi_1)\leq \varepsilon,
$$
which shows the convergence of the whole sequence $(\mathfrak{Q}_1^{k}(x_0, \xi_1))_{k \in \mathbb{N}}$ to $\mathcal{Q}_1(x_0)$, i.e.,
the optimal value of the approximate optimization problem solved at the first stage converges
with probability one to the optimal value $\mathcal{Q}_1(x_0)$ of the optimization problem.

Next, consider an accumulation point $x^*$ of the sequence $(x_1^k)_{k \in \mathbb{N}}$.
There exists a set $K$ such that the sequence $(x_1^k)_{k \in K}$ converges to $x^*$.
Recalling that $n_1$ is the node associated to the first stage, we can take 
$K_{n_1}=K$ in the proof of Proposition \ref{propconv1}  and by definition of $x_{[n_{1}]}^*$ we have $x_{[n_{1}]}^*=x^*$.
Also, by definition
of $x_1^k$,
\begin{equation}\label{convfirststage}
\begin{array}{lll}
 f_1(x_1^k, \Psi_1^k) + \mathcal{Q}_2^{k-1}(x_1^k)&=& \mathfrak{Q}_1^{k-1}(x_0, \xi_1^k)=\mathfrak{Q}_1^{k-1}(x_0, \xi_1).
\end{array} 
\end{equation}
Using $H_2(2)$, Lemma \ref{convkm1}, and the continuity of $\mathcal{Q}_2$ on $\mathcal{X}_1$, we have
$$
\lim_{k \rightarrow +\infty,\, k \in K_{n_1}}\,\mathcal{Q}_2^{k-1}(x_1^k)=\mathcal{Q}_2( x_{[n_1]}^*).
$$
Taking the limit in \eqref{convfirststage} when $k \rightarrow +\infty$ with $k \in K_{n_1}$, we obtain
$$
\mathcal{Q}_1(x_0)= f_1(x^*, \Psi_1) + \mathcal{Q}_2(x^* )=F_{1}(x^{*}, \Psi_1).
$$
Since for every $k \in K_{n_1}$,  $x_{1}^k$ is feasible for the first stage problem, so is $x^*$ (due to the lower semicontinuity of $g_1(x_0,\cdot,\Psi_1)$, $X_1(x_0, \xi_1)$ is closed) and 
$x^*$ is an optimal solution to the first stage problem.\hfill
\end{proof}

}\fi

\if{

\section{Complexity analysis}\label{comp}

We start our analysis considering two-stage stochastic programs, i.e., $T=2$.
In this case, following the proof of convergence of Kelley's cutting plane 
algorithm \cite{ruznooknl}[Theorem 7.7], fixing $\varepsilon>0$ and introducing
$$
\mathcal{K}_{\varepsilon}=\{k \in \mathbb{N}\;:\;F_1(x_{0:1}^k, \Psi_1) \geq \mathcal{Q}_1(x_0)+\varepsilon \},
$$
we can show that if $k_1<k_2 \in \mathcal{K}_{\varepsilon}$ then
\begin{equation}\label{complexity1}
\|x_1^{k_2}-x_1^{k_1}\|_2 \geq \frac{\varepsilon}{2 M_0}
\end{equation}
where $M_0$ is an upper bound on the norm of subgradients $\pi_{1 k j}$.
For a bounded set $X$ and $r>0$, let 
$\mathcal{N}_1 \Big( X, r\Big)$ be the maximal number of points
in the set $X$ such that the distance between any two of these points is at least $r$.
From \eqref{complexity1}, it follows that 
$1+\mathcal{N}_1 \Big( X_1(x_0, \xi_1), \frac{\varepsilon}{2 M_0}\Big)$ 
gives an upper bound on the number of iterations required to obtain an $\varepsilon$-optimal
first stage solution. Since $X_1(x_0, \xi_1) \subset \mathbb{R}^{n}$ is bounded, 
let $+\infty>R>0$ be such that 
$B_n(0,R)$ is the ball in $\mathbb{R}^n$
centered at the origin of smallest radius containing $X_1(x_0, \xi_1)$.
Then $1+\mathcal{N}_1 \Big( B_n(0,R), \frac{\varepsilon}{2 M_0}\Big)$ is also 
an upper bound on the number of iterations required to obtain an $\varepsilon$-optimal
first stage solution. To see how this upper bound depends on the dimension $n$ of $x_1$, we
introduce for $0<r \leq R$  the maximal number $\mathcal{N}_1 \Big(X, r\Big)$ of 
non-overlapping balls of radius $r$ than can be placed
in $X$. Recalling definition \eqref{epsfatten} of $\Big[ X\Big]^{\varepsilon}$ for 
$\varepsilon>0$, we have
\begin{equation}\label{complexity2}
\mathcal{N}_1 \Big(X, r\Big)=\mathcal{N}_2 \Big( \Big[ X\Big]^{\frac{r}{2}}, \frac{r}{2}\Big).
\end{equation}
Indeed, since we can find $\mathcal{N}_1 \Big(X, r\Big)$ points in
$X$ such that the distance between any two of these points is at least $r$
and since the balls centered at these points with radius $r/2$ are non-overlapping
balls contained in $\Big[ X\Big]^{\frac{r}{2}}$, we have 
$\mathcal{N}_2 \Big(\Big[ X\Big]^{\frac{r}{2}}, \frac{r}{2}\Big) \geq \mathcal{N}_1 \Big(X, r\Big)$.
Conversly, the centers of any set of non-overlapping
balls of radius $r/2$ contained in $\Big[ X\Big]^{\frac{r}{2}}$ form a set of points
in $X$ such that the distance between any two of these points is at least
$r$ which implies $\mathcal{N}_1 \Big(X, r\Big) \geq \mathcal{N}_2 \Big(\Big[ X\Big]^{\frac{r}{2}}, \frac{r}{2}\Big)$.
We have thus justified that \eqref{complexity2} holds. Now take a set of non-overlapping balls 
of radius $0<\frac{r}{2}\leq R$ contained in the ball $B_n\Big(0,R+\frac{r}{2}\Big)$.
The volume of these non-overlapping balls is bounded from above
by the volume of the ball $B_n\Big(0,R+\frac{r}{2}\Big)$. 
Taking $X=B_n(0,R)$ in \eqref{complexity2} gives
$\mathcal{N}_1 \Big(B_n(0,R), r\Big)=\mathcal{N}_2 \Big( B_n(0, R+\frac{r}{2}), \frac{r}{2}\Big)$
and denoting by
$\omega_n$ the volume of the unit ball $B_n(0,1)$, we obtain for
$\mathcal{N}_1 \Big(B_n(0,R), r\Big)$ the upper bound
\begin{equation}\label{upperboundcomplexity}
E\left[\frac{\omega_n \left( R+\frac{r}{2} \right)^{n}} {\omega_n \left( \frac{r}{2} \right)^n}\right]
=E\left[ \left( 1 +\frac{2R}{r} \right)^{n} \right].
\end{equation}
It follows that
$$
1+E\left[\left(1+\frac{4 R M_0}{\varepsilon}\right)^n  \right]
$$
is an upper bound on the number of iterations required to obtain an $\varepsilon$-optimal
first stage solution. We see that this upper bound increases exponentially with the dimension $n$
of the first stage decision $x_1$.

When $n=1$, we easily see that $\mathcal{N}_2\Big(B_1(0,R+\frac{r}{2}), \frac{r}{2}\Big)=1+E[\frac{2R}{r}]$
which corresponds to upper bound \eqref{upperboundcomplexity} for $n=1$: the upper bound
\eqref{upperboundcomplexity}
is thus tight for $n=1$. For $n=2$, upper bound \eqref{upperboundcomplexity} is not tight anymore. Indeed, for $r=R$
and $n=2$,
we have $\mathcal{N}_1 \Big(B_2(0,R), R\Big)=7$ obtained with the configuration
given in Figure \ref{fig1} whereas upper bound \eqref{upperboundcomplexity} is 9 in this case. 
\begin{figure}
\begin{tabular}{l}
\input{Complexity1.pstex_t}
\end{tabular}
\caption{Maximal number of points with mutual distance at least $R$ in a circle of radius $R$
in $\mathbb{R}^2$.}
\label{fig1}
\end{figure}
More generally, when $n=2$, our conjecture is that 
$\mathcal{N}_2 \Big(B_2(0, R+\frac{r}{2}), \frac{r}{2}\Big)$
can be obtained
taking extremities of non-overlapping equilateral
triangles as shown in Figure \ref{fig2}.
\begin{figure}
\begin{tabular}{l}
\input{Complexity2.pstex_t}
\end{tabular}
\caption{Set of points with mutual distance at least $0<r<R$ in a circle of radius 
$R$ in $\mathbb{R}^2$ ($3r<R<4r$ in the figure above).}
\label{fig2}
\end{figure}
Assuming $kr \leq R <(k+1)r$, i.e., setting $k=E[\frac{R}{r}] \geq 1$,
we place one point at the origin, 
$6 \ell$ new points on the $\ell$-th layer of triangles
for $1 \leq \ell \leq k$ and a number of points from the last $(k+1)$-th layer
that depends on $R$. Let us compute the number of points from this last
layer. Let $\Delta OAE$ be the equilateral triangle with side lengths
$(k+1)r$ and with vertices the origin $O=[0;0]$,
$A=[(k+1)r;0]$, and $E=[\frac{(k+1)r}{2}; \frac{\sqrt{3}}{2}(k+1)r]$ (see Figure \ref{fig2}).
Denoting by $C$ the midpoint of line segment $\overline{AE}$, there is no point from
the $(k+1)$-th layer inside the disk $B_2(0,R)$ if $R<OC=\frac{\sqrt{3}}{2} (k+1) r$.
As a result, the total number of points is
\begin{equation} \label{circleform1}
1+6 \displaystyle \sum_{\ell=1}^{E[\frac{R}{r}]} \ell = 1+3 E[\frac{R}{r}](1+E[\frac{R}{r}]) \mbox{ if }kr \leq R<\frac{\sqrt{3}}{2} (k+1) r
\end{equation}
in this case. We obtain 1+$3k(1+k)$ points and observe that
upper bound \eqref{upperboundcomplexity} is strictly larger:
$E[\Big(1+\frac{2 R}{r}\Big)^2 ] \geq E[\Big(1+2k\Big)^2 ]=1+4k(1+k)>1+3k(1+k)$.

If, on the contrary, $\frac{\sqrt{3}}{2} (k+1) r \leq R < (k+1)r$,
let $B$ and $D$ be the  intersection of $\overline{AE}$ and the circle
of center the origin $O$ and of radius $R$. Note that $B$ and $D$ belong to
the interior of $\overline{AE}$ and that the number of points to add
from the $(k+1)$-th layer is, by symmetry, 6 times the number $P$ of points of
this layer that belong to line segment $\overline{BD}$.
Alternatively, $P$ is $k+2$ minus twice the number $\lceil \frac{AB}{r}\rceil$ of points of this layer 
that are
in the interior of line segment $\overline{AB}$ but different from B. 
It remains to compute $AB$. We have $OA=AE=(k+1)r$, $AC=(k+1)r/2$
and since $\Delta OCA$ is a right triangle with $\angle AOC=\pi/6$, we get
$OC=OA \cos(\pi/6)=\frac{\sqrt{3}}{2}(k+1)r$
and 
using Pythagoras' theorem in $\Delta OBC$,  
$$
\begin{array}{lll}
BC&=&\sqrt{OB^2-OC^2}=\sqrt{R^2 - \frac{3}{4}(k+1)^2 r^2}, \mbox{ and}\\
AB&=&AC-BC=\frac{(k+1)r}{2}-\sqrt{R^2 - \frac{3}{4}(k+1)^2 r^2}.
\end{array}
$$
The total number of points is then
\begin{equation}\label{nbmax}
\begin{array}{lll}
1+6 \displaystyle \sum_{\ell=1}^{E[\frac{R}{r}]} \ell +6\left(k+2-2 \lceil \frac{AB}{r}\rceil  \right)   \\
=1+3 E[\frac{R}{r}](1+E[\frac{R}{r}]) +  6\left(k+2-2\left\lceil \frac{k+1}{2}-\sqrt{\frac{R^2}{r^2} - \frac{3}{4}(k+1)^2} \right\rceil \right)
\end{array}
\end{equation}
which is strictly smaller than upper bound \eqref{upperboundcomplexity} for
$\frac{\sqrt{3}}{2} (k+1) r \leq R < (k+1)r$: 
indeed, for $\frac{\sqrt{3}}{2} (k+1) r \leq R < (k+1)r$ we have
$\frac{AB}{r} = \frac{k+1}{2}-\sqrt{\frac{R^2}{r^2} - \frac{3}{4}(k+1)^2}>0$
(it is clear geometrically, see Figure \ref{fig2})
which implies 
$\left\lceil \frac{k+1}{2}-\sqrt{\frac{R^2}{r^2} - \frac{3}{4}(k+1)^2} \right\rceil \geq 1$ and
\eqref{nbmax} is strictly bounded from above by upper bound \eqref{upperboundcomplexity}:
$$
\begin{array}{lll}
1+3 k (k+1)+6 k  & \leq & 1 + 3(k+1)^2 + 3(k+1)\\
& \leq & 1 + 3(k+1)^2 + E[2 \sqrt{3} (k+1)] =E[(1+\sqrt{3}(k+1))^2] \\
& \leq & E[\Big(1+\frac{2 R}{r}\Big)^2 ] \;\;\;\mbox{since }\frac{\sqrt{3}}{2} (k+1) r \leq R.
\end{array}
$$
Now let $B=\{x \in \mathbb{R}^{n}\;:\; {\underline x} \leq x \leq {\bar x}\}$ be the smallest box
containing $X_1(x_0, \xi_1)$.
Then $1+\mathcal{N}_1 \Big( B, \frac{\varepsilon}{2 M_0}\Big)$ is also 
an upper bound on the number of iterations required to obtain an $\varepsilon$-optimal
first stage so
lution.
Setting $r=\frac{\varepsilon}{2 M_0}$, recalling that
$\mathcal{N}_1 \Big(B, r\Big)=\mathcal{N}_2 \Big( \Big[ B\Big]^{\frac{r}{2}}, \frac{r}{2}\Big)$
and observing that $\Big[ B\Big]^{\frac{r}{2}}\subseteq \{x \in \mathbb{R}^{n}\;: \;{\underline x}-\frac{r}{2} \leq x \leq {\bar x}+\frac{r}{2}\}$, 
we have for $\mathcal{N}_1 \Big(B, r\Big)$ the upper bound
$
E\Big[\frac{ \Pi_{i=1}^{n} ({\bar x}_i - {\underline x}_{i} +r)}  {\omega_n \left( \frac{r}{2}\right)^n}\Big]
$
and in the case when ${\bar x}_i=-{\underline x}_i=R>0$, 
$$
E\Big[\frac{1}{\omega_n}\left[2\left(1+\frac{4 R M_0}{\varepsilon}\right)\right]^n \Big]
$$
is an upper bound-still increasing exponentially with $n$-on the number of
iterations required to obtain an $\varepsilon$-optimal first stage solution.
For $n=2$, the upper bound
$E\Big[\frac{ \Pi_{i=1}^{n} ({\bar x}_i - {\underline x}_{i} +r)}  {\omega_n \left( \frac{r}{2}\right)^n}\Big] =E\Big[\frac{  4 ({\bar x}_2 -{\underline x}_2)({\bar x}_1 -{\underline x}_1 +r)}{\pi r^2}\Big]$
for $\mathcal{N}_1 \Big(B, r\Big)$ can be improved placing 
$$
\left(1+E\left[\frac{{\bar x}_1 -{\underline x}_1}{r}\right]\right)\left(1+E\left[\frac{2({\bar x}_2 -{\underline x}_2)}{\sqrt{3} r}\right]\right)
$$
balls in the box as shown in Figure \ref{fig3}.

\begin{figure}[H]
\begin{center}
\begin{tabular}{l}
\input{Box2.pstex_t}
\end{tabular}
\caption{Set of points with mutual distance at least $r>0$ in a box of  
$\mathbb{R}^2$ of surface $\ell L$.}
\label{fig3}
\end{center}
\end{figure}

}\fi

\section*{Acknowledgments} The author's research was 
partially supported by an FGV grant, CNPq grant 307287/2013-0, 
FAPERJ grants E-26/110.313/2014, and E-26/201.599/2014.
We would like to thank the two anonymous reviewers for their suggestions and comments.

\addcontentsline{toc}{section}{References}
\bibliographystyle{plain}
\bibliography{Risk_Averse_SDDP}

\end{document}